\documentclass[11pt]{amsart}
\usepackage[paperheight=11in,paperwidth=8.5in,left=1in,top=1in,right=1in,bottom=1in,heightrounded]{geometry}
\usepackage[pdfstartpage={},pdfstartview={FitH},pdftitle={},pdfauthor={},pdfsubject={},pdfcreator={},pdfproducer={},pdfkeywords={},pagebackref=false,colorlinks=true,linkcolor={red!87.5!black},citecolor={blue!50!black},linktocpage=true]{hyperref}

\usepackage{accents}
\usepackage{afterpage}
\usepackage{amssymb}
\usepackage{array}
\usepackage{bm}
\usepackage{bbm}
\usepackage{braket}
\usepackage{calc}
\usepackage[margin=0cm,font=small,labelfont=bf]{caption}
\usepackage[noadjust]{cite}
\usepackage{centernot}
\usepackage{dsfont}
\usepackage{enumitem}
\usepackage{float}
\restylefloat{figure}
\RequirePackage{graphicx}
\usepackage{indentfirst}
\usepackage{mathdots}
\usepackage{mathrsfs}
\usepackage{mathtools}
\usepackage{mleftright}
\usepackage{pgf}
\usepackage{scalerel}
\usepackage{stackengine}
\usepackage{tikz}
\usetikzlibrary{arrows,shapes}
\usepackage{tikz-cd}

\theoremstyle{plain}
\newtheorem{theorem}{Theorem}[section]
\newtheorem{lemma}[theorem]{Lemma}
\newtheorem{proposition}[theorem]{Proposition}
\theoremstyle{remark}
\newtheorem{definition}[theorem]{Definition}
\newtheorem{example}[theorem]{Example}
\newtheorem{remark}[theorem]{Remark}

\DeclareMathOperator{\real}{Re}

\DeclareMathOperator{\mob}{M\ddot{o}b}
\newcommand{\cov}{\mathrm{Cov}}
\newcommand{\N}{\mathbb{N}}

\newcommand{\R}{\mathbb{R}}
\newcommand{\C}{\mathbb{C}}

\newcommand{\sphr}{\mathbb{S}}
\newcommand{\pr}{\mathbb{P}}
\newcommand{\E}{\mathbb{E}}
\newcommand{\wto}{\stackrel{\op{w}}{\to}}
\newcommand{\ssim}{\mathord{\sim}}
\newcommand{\op}[1]{\operatorname{#1}}
\newcommand{\mcal}[1]{\mathcal{#1}}
\newcommand{\mfk}[1]{\mathfrak{#1}}
\newcommand{\msr}[1]{\mathscr{#1}}

\newcommand{\mbf}[1]{\mathbf{#1}}
\newcommand{\indc}[1]{\mathds{1}\left\{#1\right\}}

\newcommand{\inn}[2]{\langle #1, #2 \rangle}

\newcommand{\into}{\hookrightarrow}
\newcommand{\norm}[1]{\lVert#1\rVert}
\newcommand{\trace}{\mathrm{Tr}}
\newcommand{\vout}{V_{\op{out}}}
\newcommand{\vin}{V_{\op{in}}}
\newcommand{\pin}{\pi_{\op{in}}}
\newcommand{\pim}{\pi_{\op{mix}}}
\newcommand{\pout}{\pi_{\op{out}}}
\newcommand{\calP}{\mathcal{P}}

\newcommand{\calK}{\mathcal{K}}

\newcommand{\ubp}{\mcal{UBP}}
\newcommand{\bK}{\mbf{K}}
\newcommand{\ten}{\mbf{T}_{d, N}}
\newcommand{\nc}{\mathcal{NC}}

\DeclareMathOperator{\per}{per}
\newcommand{\sym}{\mathrm{Sym}}
\DeclareMathOperator{\sign}{sgn}
\newcommand{\tdn}[2]{\mbf{T}_{d, N}(#1|#2)}
\newcommand{\un}[2]{u_N(#1|#2)}

\begin{document}
\title[Contracted tensor ensembles]{Spectral asymptotics for contracted tensor ensembles}

\author[Benson Au]{Benson Au$^\dagger$}
\thanks{$^\dagger$Department of Statistics, University of California, Berkeley, \href{mailto:bensonau@berkeley.edu}{bensonau@berkeley.edu}}
\address{Department of Statistics\\
         University of California, Berkeley\\
         367 Evans Hall \# 3860\\
         Berkeley, CA 94720-3860\\
         USA}
       \email{\href{mailto:bensonau@berkeley.edu}{bensonau@berkeley.edu}}

\author[Jorge Garza-Vargas]{Jorge Garza-Vargas$^\ddagger$}
\thanks{$^\ddagger$CMS, California Institute of Technology, \href{mailto:jgarzav@caltech.edu}{jgarzav@caltech.edu}}
\thanks{JGV was supported in part by NSF grant CCF-2009011.}
\address{Department of Computing and Mathematical Sciences\\
         California Institute of Technology\\
         1200 E.\@ California Blvd., MC 305-16\\
         Pasadena, CA 91125-2100\\
         USA}
       \email{\href{mailto:jgarzav@caltech.edu}{jgarzav@caltech.edu}}

\date{\today}

\begin{abstract}\label{sec:abstract}
Let $\mbf{T}_{d, N}: \Omega \to \R^{N^d}$ be a random real symmetric Wigner-type tensor. For unit vectors $(u_N^{(i, j)})_{i \in I, j \in [d-2]} \subset \sphr^{N-1}$, we study the contracted tensor ensemble
\[
    \left(\frac{1}{\sqrt{N}}\mbf{T}_{d, N}\left[u_N^{(i, 1)} \otimes \cdots \otimes u_N^{(i, d-2)}\right]\right)_{i \in I}.
\]
For large $N$, we show that the joint spectral distribution of this ensemble is well-approximated by a semicircular family $(s_i)_{i \in I}$ whose covariance $(\bK_{i, i'}^{(N)})_{i, i'\in I}$ is given by the rescaled overlaps of the corresponding symmetrized contractions
\[
    \mbf{K}_{i, i'}^{(N)} = \frac{1}{d(d-1)}\inn{u_N^{(i, 1)} \odot \cdots \odot u_N^{(i, d-2)}}{u_N^{(i', 1)} \odot \cdots \odot u_N^{(i', d-2)}},
\]
which is the true covariance of the ensemble up to a $O_d(N^{-1})$ correction. We further characterize the extreme cases of the variance $\mbf{K}_{i, i}^{(N)} \in [\frac{1}{d!}, \frac{1}{d(d-1)}]$. Our analysis relies on a tensorial extension of the usual graphical calculus for moment method calculations in random matrix theory, allowing us to access the independence in our random tensor ensemble.
\end{abstract}

\maketitle
\setcounter{tocdepth}{2}
\tableofcontents

\section{Introduction}\label{sec:intro}

The storied history of random matrix theory (RMT) has its beginnings in applications: first, in multivariate statistics \cite{Wis28}; then, famously, in nuclear physics \cite{Wig55}. The spectacular progress of RMT in the interim, documented for example by \cite{Meh04,AGZ10,BS10,EY17}, shows a flourishing mathematical discipline with ever-increasing ties to the applied sciences. Indeed, by now, RMT techniques comprise a robust and increasingly influential set of tools for a wide range of disparate fields \cite{ABDF11,BDS16,MS17}.

In modern applications, one often encounters higher order structures. In this generality, a matrix is replaced by a tensor \cite{Lim21}.

\begin{definition}[Tensor]\label{defn:tensor}
A \emph{(real) $d$-th order tensor $\mbf{A}_{d, N_1, \ldots, N_d}$ of dimension $(N_1, \ldots, N_d)$} is a tuple
\[
    \big(\mbf{A}_{d, N_1, \ldots, N_d}(k_1, \ldots, k_d)\big)_{(k_1, \ldots, k_d) \in [N_1] \times \cdots \times [N_d]} \in \R^{N_1 \cdots N_d} \cong \R^{N_1} \otimes \cdots \otimes \R^{N_d}.
\]
A tensor of the form $\mbf{A}_{d, N_1, \ldots, N_d} = v_1 \otimes \cdots \otimes v_d$ for $v_i \in \R^{N_i}$ is said to be \emph{pure}. We write $\mcal{T}_{d, N_1, \ldots, N_d}$ for the inner product space of all $d$-th order tensors of dimension $(N_1, \ldots, N_d)$. If $N_1 = \cdots = N_d = N$, then we abbreviate the notation to $\mbf{A}_{d, N}$ and $\mcal{T}_{d, N}$ respectively.

For $\sigma \in \mfk{S}_{d}$ and $\mbf{A}_{d, N} \in \mcal{T}_{d, N}$, we define the permuted tensor $\mbf{A}_{d, N}^\sigma \in \mcal{T}_{d, N}$ by
\[
    \mbf{A}_{d, N}^\sigma(k_1, \ldots, k_d) := \mbf{A}_{d, N}(k_{\sigma(1)}, \ldots, k_{\sigma(d)}). 
\]
A tensor $\mbf{A}_{d, N}$ is \emph{symmetric} if it is invariant under permutation:
\[
    \mbf{A}_{d, N}^\sigma = \mbf{A}_{d, N}, \quad \forall \sigma \in \mfk{S}_d.
\]
Similarly, a tensor $\mbf{A}_{d, N}$ is \emph{antisymmetric} if
\[
    \mbf{A}_{d, N}^\sigma = \sign(\sigma) \mbf{A}_{d, N}, \quad \forall \sigma \in \mfk{S}_d.
\]
We write $\mcal{S}_{d, N}$ for the subspace of symmetric tensors and $\sym: \mcal{T}_{d, N} \to \mcal{T}_{d, N}$ for the symmetrization operator
\[
    \sym(\mbf{A}_{d, N}) = \frac{1}{d!}\sum_{\sigma \in \mfk{S}_d} \mbf{A}_{d, N}^\sigma,
\]
which is a projection $\sym = \sym^2 = \sym^*$ onto $\mcal{S}_{d, N}$. For pure tensors, we use the notation
\[
    v_1 \odot \cdots \odot v_d := \sym(v_1 \otimes \cdots \otimes v_d) = \frac{1}{d!}\sum_{\sigma \in \mfk{S}_d} v_{\sigma(1)} \otimes \cdots \otimes v_{\sigma(d)}.
\]
\end{definition}

The spectral theory of symmetric tensors for $d \geq 3$ can be quite different from the familiar case of symmetric matrices $d = 2$. Indeed, simply finding the right definitions in this setting already constitutes a major conceptual challenge. We assume hereafter that $d \geq 3$. We focus on one particular definition of eigenvalues for tensors, formulated independently by Lim under the name \emph{$\ell^2$-eigenvalues} \cite{Lim05} and Qi under the name \emph{Z-eigenvalues} \cite{Qi05}. We refer the reader to \cite{QL17} for a comprehensive treatment of the general theory. 

\begin{definition}[Tensor operations]\label{defn:tensor_operations}
Let $\mbf{A}_{d, N}$ be a symmetric tensor. For $p \leq d$ and $\mbf{B}_{p, N} \in \mcal{T}_{p, N}$, we define the \emph{contracted tensor} $\mbf{A}_{d, N}[\mbf{B}_{p, N}] \in \mcal{S}_{d-p, N}$ by
\[
    \mbf{A}_{d, N}[\mbf{B}_{p, N}](k_1, \ldots, k_{d-p}) = \sum_{(l_1, \ldots, l_p) \in [N]^p} \mbf{A}_{d, N}(k_1, \ldots, k_{d-p}, l_1, \ldots, l_p)\mbf{B}_{p, N}(l_1, \ldots, l_p).
\]
The symmetry of $\mbf{A}_{d, N}$ implies that this definition does not depend on which coordinates we choose to contract. Consequently,
\[
    \mbf{A}_{d, N}[\sym(\mbf{B}_{p, N})] = \mbf{A}_{d, N}[\mbf{B}_{p, N}].
\]
We say that $\mbf{A}_{d, N}$ has an \emph{eigenvalue} at $\lambda \in \R$ if there exists a unit vector $u_N \in \sphr^{N-1}$ such that
\[
    \mbf{A}_{d, N}[u_N^{\otimes d-1}] = \lambda u_N,
\]
in which case we call $u_N$ an \emph{eigenvector}.
\end{definition}

This definition of an eigenpair $(\lambda, u_N) \in \R \times \sphr^{N-1}$ can be related to the usual definition for matrices through the following straightforward observation:
\[
  \mbf{A}_{d, N}[u_N^{\otimes d-1}] = \lambda u_N \iff \mbf{A}_{d, N}[u_N^{\otimes d-2}]u_N = \lambda u_N.
\]
In other words, $\lambda$ is an eigenvalue of the real symmetric tensor $\mbf{A}_{d, N}$ with eigenvector $u_N$ iff $\lambda$ is an eigenvalue of the real symmetric matrix $\mbf{A}_{d, N}[u_N^{\otimes d-2}]$ with eigenvector $u_N$. One can then hope to analyze the spectral properties of a tensor via appropriate contractions.

For a random symmetric tensor $\mbf{T}_{d, N}: \Omega \to \mcal{S}_{d, N}$, de Morais Goulart, Couillet, and Comon defined the \emph{contraction ensemble of $\mbf{T}_{d, N}$} as the family of random matrices
\[
  \mcal{M}(\mbf{T}_{d, N}) = \left(\mbf{T}_{d, N}[u_N^{\otimes d-2}]\right)_{u_N \in \sphr^{N-1}}.
\]
For $d = 3$ and $\mbf{T}_{3, N}$ distributed according to the Gaussian orthogonal tensor ensemble (GOTE), they proved that the empirical spectral distribution of $\frac{1}{\sqrt{N}}\mbf{T}_{3, N}[u_N]$ converges weakly almost surely to the semicircle distribution of variance $\frac{1}{6}$ in the large $N$ limit for any sequence of unit vectors $(u_N)_{N \in \N}$ \cite[Theorem 6]{dCC21}. In particular, the variance of the limiting semicircle does not depend on the choice of $(u_N)_{N \in \N}$. Their proof relies on an analysis of the resolvent using Stein's lemma, exploiting the Gaussian nature of the entries and the tractability of small order $d = 3$. Naturally, we are led to consider the questions of
\begin{enumerate}[label=(Q\arabic*)]
    \item \label{question:higher_order} higher order $d \geq 4$;
    \item \label{question:universality} universality for general tensor distributions;
    \item \label{question:general_contractions} general contractions $u_N^{(1)} \otimes \cdots \otimes u_N^{(d-2)} \neq u_N^{\otimes d-2}$;
    \item \label{question:joint_distribution} the joint spectral distribution of the full family of matrices in the contracted tensor ensemble with general contractions.
\end{enumerate}

We emphasize the complications that already arise just in \ref{question:higher_order} and \ref{question:universality}. While the entries of the GOTE are independent up to the symmetry constraint, the same is not true in general for the contracted GOTE $\mbf{T}_{d, N}[u_N^{(1)} \otimes \cdots \otimes u_N^{(d-2)}]$. For $d = 3$, the contraction operation can only introduce dependence between entries of $\mbf{T}_{3, N}[u_N]$ belonging to the same row or the same column; for $d \geq 4$, any two entries of the contracted GOTE can now be dependent. Moreover, the resulting correlation structure does not fit into the universality scheme of \cite{BMP15}, which would have allowed one to deduce the general result from the Gaussian case. See Section \ref{sec:existing_results} for the precise details. To these ends, we consider an analogue of Wigner matrices for tensors.

\begin{definition}[Wigner tensor]\label{defn:wigner_tensor}
Let $\big(X_N(k_1, \ldots, k_d)\big)_{N \in \N, 1 \leq k_1 \leq \cdots \leq k_d \leq N}$ be a family of independent random variables such that
\begin{enumerate}[label=(\roman*)]
\item \label{cond:centered} the off-diagonal entries $\#(\{k_1, \ldots, k_d\}) \neq 1$ are centered;
\item \label{cond:gote_variance} the entries with at least three distinct indices $\#(\{k_1, \ldots, k_d\}) \geq 3$ have variance $\binom{d}{b_1, \ldots, b_N}^{-1}$, where $b_l = \#(\{s \in [d] : k_s = l\})$ is the multiplicity of $l \in [N]$ as an index;
\item \label{cond:moment_bound} we have a strong uniform control on the moments:
\begin{equation}\label{eq:moment_bound}
  \sup_{N \in \N} \sup_{1 \leq k_1 \leq \cdots \leq k_d \leq N} \E\big[|X_N(k_1, \ldots, k_d)|^m\big] \leq C_m < \infty, \quad \forall m \in \N.
\end{equation}
\end{enumerate}
We call the random symmetric tensor defined by
\[
  \mbf{T}_{d, N}(k_1, \ldots, k_d) = X_N(k_{i_1}, \ldots, k_{i_d}), \quad \forall k_1, \dots, k_d \in [N]  
\]
an \emph{unnormalized $d$-th order Wigner tensor}, where $k_{i_1} \leq \cdots \leq k_{i_d}$ is any nondecreasing ordering of the indices $k_1, \ldots, k_d$. Hereafter, when we refer to a tensor $\mbf{T}_{d, N}$, we implicitly refer to a sequence of tensors $(\mbf{T}_{d, N})_{N \in \N}$.
\end{definition}

\begin{remark}\label{rem:gote_variance}
The variance profile in \ref{cond:gote_variance} is chosen to match the $\op{GOTE}(d, N)$, which is defined by the density
\[
  f_{d, N}(T) = \frac{1}{Z_d(N)} e^{-\norm{T}_2^2/2}, \quad T \in \mcal{S}_{d, N}.
\]
Here, $\norm{\cdot}_2$ is the usual Euclidean norm on $\R^{N^d}$ and $Z_d(N)$ is a normalizing constant. The condition in \ref{cond:gote_variance} omits the entries with at most two distinct indices because they do not contribute to the leading order term in the trace expansion. Our results extend to any variance profile that is constant on the integer partition of $d$ determined by $(b_l)_{l \in [N]}$; however, in general, the resulting covariance $(\mbf{K}_{i, i'}^{(N)})_{i, i' \in I}$ no longer admits a simple formulation as below (cf.\ Section \ref{sec:general_variance_profile}).
\end{remark}

For any sequence of families $(u_N^{(i, j)})_{i \in I, j \in [d-2]} \subset \sphr^{N-1}$, we define $\bK^{(N)}=(\mbf{K}_{i, i'}^{(N)})_{i, i' \in I}$ as the rescaled Gram matrix of the symmetrized pure tensors $(u_N^{(i, 1)} \odot \cdots \odot u_N^{(i, d-2)})_{i \in I}$:
\begin{equation}\label{eq:gram_matrix}
  \mbf{K}_{i, i'}^{(N)} = \frac{1}{d(d-1)} \inn{u_N^{(i, 1)} \odot \cdots \odot u_N^{(i, d-2)}}{u_N^{(i', 1)} \odot \cdots \odot u_N^{(i', d-2)}}.
\end{equation}
We can now state our main result in answer to questions \ref{question:higher_order}-\ref{question:joint_distribution}.

\begin{theorem}\label{thm:contracted_tensor_ensemble}
Let $\mbf{T}_{d, N}$ be a Wigner tensor. Then the corresponding contracted tensor ensemble $\left(\frac{1}{\sqrt{N}}\mbf{T}_{d, N}\left[u_N^{(i, 1)} \otimes \cdots \otimes u_N^{(i, d-2)}\right]\right)_{i \in I}$ converges in noncommutative distribution almost surely iff the limits
\[
     \mbf{K}_{i, i'} = \lim_{N \to \infty} \mbf{K}_{i, i'}^{(N)}
\]
exist for every $i, i' \in I$, in which case the limit object is a semicircular family $(s_i)_{i \in I}$ of covariance $(\mbf{K}_{i, i'})_{i, i' \in I}$. In particular, the empirical spectral distribution of the single matrix model $\frac{1}{\sqrt{N}}\mbf{T}_{d, N}\left[u_N^{(i, 1)} \otimes \cdots \otimes u_N^{(i, d-2)}\right]$ converges weakly almost surely to a semicircle distribution:
\[
    \frac{1}{N} \sum_{k = 1}^N \delta_{\lambda_k\left(\frac{1}{\sqrt{N}}\mbf{T}_{d, N}\left[u_N^{(i, 1)} \otimes \cdots \otimes u_N^{(i, d-2)}\right]\right)} \wto \frac{1}{2\pi \mbf{K}_{i, i}} (4\mbf{K}_{i, i} - x^2)_+^{1/2}\, dx.
\]
\end{theorem}

The proof of Theorem \ref{thm:contracted_tensor_ensemble} relies on a tight control of the mixed moments of the contracted tensor ensemble. To formulate this precisely, we introduce some notation. For any $i_1, \dots, i_m \in I$, we define the product
\begin{align*}
  P_N(i_1, \dots, i_m) := &\left(\frac{1}{\sqrt{N}} \mbf{T}_{d, N}\left[u_N^{(i_1, 1)} \otimes \cdots \otimes u_N^{(i_1, d-2)}\right]\right) \\
  &\cdots \left(\frac{1}{\sqrt{N}} \mbf{T}_{d, N}\left[u_N^{(i_m, 1)} \otimes \cdots \otimes u_N^{(i_m, d-2)}\right]\right).
\end{align*}
The following theorem quantifies the approximation of our tensor ensemble by a semicircular family.

\begin{theorem}\label{thm:concentrationofmixedmoments}
Let $I_0 \subset I$ be a finite subset. For any $m_0 \in \N$ and $M, \varepsilon > 0$, there is a constant $C = C(d, \#(I_0), m_0, M, \varepsilon)$ independent of $N$ and the choice of contracting vectors $(u_N^{(i, j)})_{i\in I, j \in [d-2]} \subset \sphr^{N-1}$ such that\small
\begin{equation}\label{eq:concentration_of_mixed_moments}
\mathbb{P}\left[ \max_{m\leq m_0, i_1, \dots, i_m\in I_0}  \left|\frac{1}{N}\trace\big(P_N(i_1, \dots, i_m)\big) -\varphi\big(s_{i_1}^{(N)} \cdots s_{i_m}^{(N)}\big)\right| > \varepsilon \right] \leq CN^{-M},
\end{equation}\normalsize
where $\varphi(s_{i_1}^{(N)}\cdots s_{i_m}^{(N)})$ is the mixed moment of a semicircular family $(s_i^{(N)})_{i \in I}$ of covariance $\bK^{(N)}$. 
\end{theorem}

In particular, Theorem \ref{thm:concentrationofmixedmoments} does \emph{not} assume that the limits $\lim_{N \to \infty} \mbf{K}_{i, i'}^{(N)}$ exist (cf.\ Theorem \ref{thm:contracted_tensor_ensemble}). In any case, it is still true that the joint spectral distribution of the contracted tensor ensemble is approximately multivariate semicircular in high dimensions. The only possibility precluding convergence is the joint spectral distribution accumulating around semicircular families with distinct covariance structures (rather than a different type of distribution entirely). Nevertheless, one can use the Cauchy-Schwarz inequality and the fact that $\norm{\sym} \leq 1$ to show that
\begin{equation}\label{eq:covariance_bound}
    |\mbf{K}_{i, i'}^{(N)}| \leq \frac{1}{d(d-1)}.
\end{equation}
Thus, one always has convergence to a specific semicircular family along some subsequence. See Figure \ref{figure:combined_marginals} and Figure \ref{figure:sum_anticom} for simulations.

\begin{example}[Common contracting vectors]\label{eg:common_contracting_vector}
In the case of common contracting vectors $u_N^{(i, j)} \equiv u_N^{(i)}$ and $u_N^{(i', j)} \equiv u_N^{(i')}$, the inner product in \eqref{eq:gram_matrix} amounts to
\[
    \mbf{K}_{i, i'}^{(N)} = \frac{1}{d(d-1)} \inn{u_N^{(i)}}{u_N^{(i')}}^{d-2}.
\]
In particular, the marginals of the limiting semicircular family $(s_i)_{i \in I}$ have variance $\mbf{K}_{i, i} = \frac{1}{d(d-1)}$ regardless of the choice of contracting vectors $(u_N^{(i)})_{N \in \N}$ (cf.\ \cite[Theorem 6]{dCC21}), and asymptotic orthogonality $\inn{u_N^{(i)}}{u_N^{(i')}} = o(1)$ is equivalent to asymptotic freeness $\mbf{K}_{i, i'} = 0$ (cf.\ \cite[Theorem 2.6.2]{VDN92}).
\end{example}

\begin{example}[Extreme cases for the variance]\label{eg:extreme_cases}
The previous example shows that the upper bound in \eqref{eq:covariance_bound} is tight. In fact, this is essentially the only way to achieve the upper bound. Indeed, since $\sym$ is a projection onto $\mcal{S}_{d-2, N}$,
\[
    \mbf{K}_{i, i}^{(N)} = \frac{1}{d(d-1)} \iff u_N^{(i, 1)} \odot \cdots \odot u_N^{(i, d-2)} = u_N^{(i, 1)} \otimes \cdots \otimes u_N^{(i, d-2)}.
\]
In other words, the upper bound is achieved iff the pure tensor is symmetric. For unit vectors $u_N^{(i, j)} \in \sphr^{N-1}$, this is equivalent to $u_N^{(i, j)} = \pm u_N^{(i, j')}$, a common contracting vector up to sign.

Given that the upper bound for $\mbf{K}_{i, i}^{(N)}$ is achieved iff the pure tensor is symmetric, one naturally expects a lower bound that is achieved iff the pure tensor is maximally asymmetric in some sense. To obtain such a characterization, we use the permanental representation of the symmetrized inner product \cite[Theorem 2.2]{Min78}:
\[
    \inn{u_N^{(i, 1)} \odot \cdots \odot u_N^{(i, d-2)}}{u_N^{(i', 1)} \odot \cdots \odot u_N^{(i', d-2)}}
    = \frac{1}{(d-2)!}\per\left[\left(\inn{u_N^{(i, j)}}{u_N^{(i', k)}}\right)_{j, k= 1}^{d-2}\right].
\]
When $i=i'$, this is the permanent of the Gram matrix
\[
    \mbf{G}_{N}^{(i)} = \left(\inn{u_N^{(i, j)}}{u_N^{(i, k)}}\right)_{j, k= 1}^{d-2}.
\]
An inequality of Marcus \cite[Theorem 2]{Mar64}, valid for positive semidefinite matrices, states that 
\[
    \per(\mbf{G}_N^{(i)}) \geq \prod_{j=1}^{d-2} \mbf{G}_N^{(i)}(j, j)
\]
with equality iff $\mbf{G}_N^{(i)}$ has a zero row or $\mbf{G}_N^{(i)}$ is diagonal. Since $\mbf{G}_N^{(i)}(j, j) \equiv 1$, we see that $\mbf{K}_N^{(i, i)} \geq \frac{1}{d!}$ with equality iff the contracting vectors $(u_N^{(i, j)})_{j \in [d-2]} \subset \sphr^{N-1}$ are orthonormal.
\end{example}

\begin{remark}[Kolmogorov-Smirnov distance]\label{rem:kolmogorov_smirnov}
In the single-matrix model $\#(I)=1$, the uniform bound $\bK_{i, i}^{(N)} \in [\frac{1}{d!}, \frac{1}{d(d-1)}]$ on the variances allows us to convert the moment approximation in Theorem \ref{thm:concentrationofmixedmoments} into the Kolmogorov-Smirnov distance $d_{\mathrm{KS}}$. This follows from a Taylor series expansion of the Fourier transforms in the Berry-Esseen smoothing inequality \cite[Chapter XVI.3, Lemma 2]{Fel71}. Consequently, for any $M, \varepsilon > 0$, there is a constant $C=C(d, M, \varepsilon)$ independent of $N$ and the choice of contracting vectors $(u_N^{(j)})_{j \in [d-2]} \subset \sphr^{N-1}$ such that
\[
    \pr \left[ d_{\mathrm{KS}}(\mu_N, \mu_{\mathrm{sc}}^{\bK^{(N)}}) > \varepsilon \right] \leq CN^{-M},
\]
where $\mu_N$ is the empirical spectral distribution of $\frac{1}{\sqrt{N}}\ten[u_N^{(1)}\otimes \cdots \otimes u_N^{(d-2)}]$ and $\mu_{\mathrm{sc}}^{\bK^{(N)}}$ is the semicircle distribution of variance 
\[
    \bK^{(N)} = \frac{1}{d(d-1)}\norm{u_N^{(1)} \odot \cdots \odot u_N^{(d-2)}}_2^2.
\]
\end{remark}

\begin{figure}
\makebox[406.0pt][c]{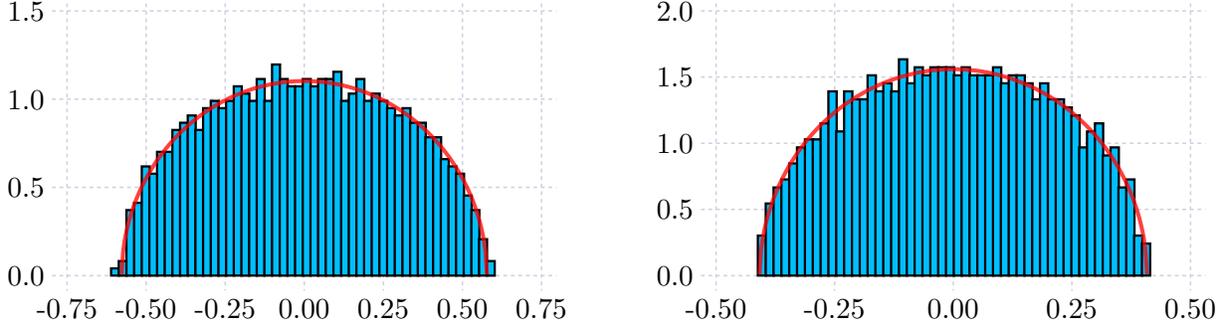}
\caption{Realizations of the empirical spectral distribution of the contracted tensor ensemble in the extreme cases of the variance covered in Example \ref{eg:extreme_cases}. We start with a single realization of a Wigner tensor $\mbf{T}_{4, 1000}$ whose upper triangular entries $k_1 \leq k_2 \leq k_3 \leq k_4$ are Rademacher distributed if $k_1$ is odd and normally distributed if $k_1$ is even. We use this to construct two contracted tensors $\mbf{A}_{1000} = \frac{1}{\sqrt{1000}}\mbf{T}_{4, 1000}[u \otimes u]$ and $\mbf{B}_{1000} = \frac{1}{\sqrt{1000}}\mbf{T}_{4, 1000}[v \otimes w]$, where $u$ is the vector in \eqref{eq:contraction_example}, $v(k) \equiv \frac{1}{\sqrt{1000}}$, and $w(k) = (-1)^{k+1}v(k)$. The eigenvalue histogram of $\mbf{A}_{1000}$ (resp., $\mbf{B}_{1000}$) is then plotted against the semicircle density of variance $\frac{1}{d(d-1)}\inn{u \odot u}{u \odot u} = \frac{1}{12}$ on the left (resp., $\frac{1}{d(d-1)}\inn{v \odot w}{v \odot w} = \frac{1}{24}$ on the right).}
\label{figure:combined_marginals}
\end{figure}

\begin{figure}
\makebox[406.0pt][c]{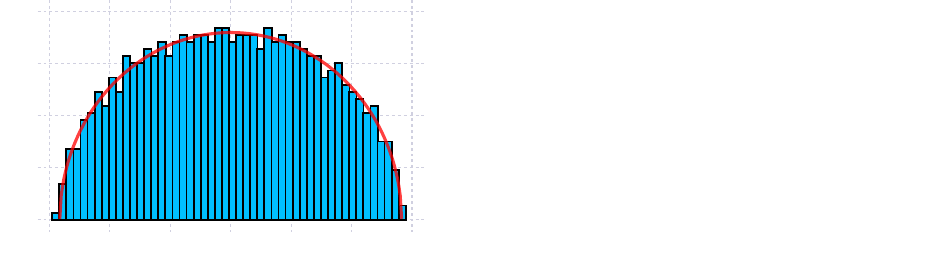}
\caption{Realizations of the joint behavior of the contracted tensor ensemble. Using the same matrices $\mbf{A}_{1000}$ and $\mbf{B}_{1000}$ from Figure \ref{figure:combined_marginals}, we construct the sum $\mbf{C}_{1000} = \mbf{A}_{1000} + \mbf{B}_{1000}$ and the anticommutator $\mbf{D}_{1000} = \mbf{A}_{1000}\mbf{B}_{1000} + \mbf{B}_{1000}\mbf{A}_{1000}$. The orthogonality $\inn{u \odot u}{v \odot w} = 0$ of the symmetrized contracting vectors implies that the approximating semicircular family in Theorem \ref{thm:concentrationofmixedmoments} is a \emph{free} semicircular family. The analytic machinery of free probability can then be used to deduce the approximations for $\mbf{C}_{1000}$ and $\mbf{D}_{1000}$. Accordingly, the eigenvalue histogram of $\mbf{C}_{1000}$ (resp., $\mbf{D}_{1000}$) is plotted against the semicircle density of variance $\frac{1}{12} + \frac{1}{24} = \frac{1}{8}$ on the left (resp., the symmetric Poisson distribution of variance $\frac{1}{144}$ \cite{NS98} on the right).}\label{figure:sum_anticom}
\end{figure}

The proof of Theorem \ref{thm:concentrationofmixedmoments} relies on the classical moment method in RMT; however, to accommodate the tensorial nature of our matrices, we augment the usual graphical framework for such calculations to include this information. Our diagrammatic approach is well-suited to the dependence structure that appears in our random matrix precisely because it allows us to access the independence in our random tensor. If one forgets the mechanism of this lineage and simply works with the matrix, then this property of the tensor is obscured. Our approach does come at the cost of additional combinatorial considerations, necessitating a bipartite representation, but we emphasize some additional advantages as well. At a high level, the graphs provide a parallel between contracted Wigner tensors and Wigner matrices, explaining the universality of the semicircle distribution for all orders $d$. After the initial investment in the single-matrix case, the multi-matrix case follows essentially from the Cauchy-Schwarz inequality. The surprisingly elegant form of the covariance in \eqref{eq:gram_matrix} can then be easily read off from the graphs. In contrast, the dependence of the covariance on the shape of the contracting vectors greatly complicates the usual analytic approach via the resolvent, even in the case of $d = 3$ \cite[Appendix A]{dCC21}. Finally, we expect that our framework can be applied to other structured tensor calculations. We present this framework and review the relevant aspects from free probability in Section \ref{sec:background}. The proofs of the main results are given in Section \ref{sec:proof_of_theorem} with the proofs of the technical results deferred to Section \ref{sec:technical_results}.

\subsubsection*{Notation}\label{sec:notation} For convenience, we adopt the following convention for big $O$ notation. Asymptotics will always be in the large $N$ limit. If the implicit constant depends on some parameters, then we indicate this with subscripts. For example, the asymptotic $O_d(N^{-1})$ in the abstract indicates a dependence of the implicit constant on the order $d$ of the tensor.

\subsection{Extensions}\label{sec:extensions}
For the sake of brevity, we restrict our proofs to Wigner tensors as defined in Definition \ref{defn:wigner_tensor} and contractions by pure tensors. Here, we collect a number of straightforward extensions.

\subsubsection{General variance profiles}\label{sec:general_variance_profile}
Instead of the GOTE variance profile in \ref{cond:gote_variance}, suppose that the entries with at least three distinct indices have variance $\sigma_\lambda^2$, where $\lambda \vdash d$ is the integer partition of $d$ determined by $(b_l)_{l \in [N]}$. To keep track of how the dependence of our tensor entries propagates through repeated contractions, we use a combinatorial object known as a uniform block permutation \cite{OEIS}.

\begin{definition}[Uniform block permutation]\label{defn:uniform_block_permutation}
Consider two disjoint copies $[n]_L$ and $[n]_R$ of the set $[n] = \{1, \ldots, n\}$. We think of $[n]_L$ and $[n]_R$ as a left side and a right side respectively. A \emph{uniform block permutation} $\pi$ of $[n]$ is a partition of the set $[n]_L \sqcup [n]_R$ such that each block $B = B_L \sqcup B_R \in \pi$ contains the same number of elements from each side $\#(B_L) = \#(B_R)$, where $B_L = B \cap [n]_L$ and $B_R = B \cap [n]_R$. We write $\mathcal{UBP}(n)$ for the set of uniform block permutations of $[n]$.
\end{definition}

For a uniform block permutation $\pi = \{B_1, \ldots, B_{\#(\pi)}\} \in \mcal{UBP}(d-2)$, we define the partition
\[
    \lambda_\pi := \left(\frac{\#(B_1)}{2}, \ldots, \frac{\#(B_{\#(\pi)})}{2}, 1, 1\right) \vdash d,
\]
where $\#(B_1) \geq \cdots \geq \#(B_{\#(\pi)})$. Theorems \ref{thm:contracted_tensor_ensemble} and \ref{thm:concentrationofmixedmoments} still hold for a general variance profile $(\sigma_\lambda^2)_{\lambda \vdash d}$, but with a less tractable formula for the covariance\small
\begin{equation}\label{eq:general_covariance}
    \calK_{i, i'}^{(N)} = \sum_{\pi \in \mcal{UBP}(d-2)} \sigma_{\lambda_\pi}^2 \left(\sum_{\phi: \pi \hookrightarrow [N]} \prod_{B \in \pi} \prod_{l \in B_L} u_N^{(i, l)}(\phi(B)) \prod_{r \in B_R} u_N^{(i', r)}(\phi(B))\right).
\end{equation}\normalsize

\begin{example}[A constant variance profile]\label{eg:constant_variance_profile}
Even in the case of the constant variance profile $\sigma_\lambda^2 \equiv 1$, the covariance in \eqref{eq:general_covariance} can be difficult to compute. If $d = 3$, then the covariance amounts to a single inner product
\[
  \calK_{i, i'}^{(N)} = \inn{u_N^{(i)}}{u_N^{(i')}}.
\]
In particular, the marginals of the limiting family $(s_i)_{i \in I}$ are always standard semicircular $\calK_{i, i} = 1$ and asymptotic orthogonality is still equivalent to asymptotic freeness. This is no longer the case for $d \geq 4$ even if we restrict to common contracting vectors. For example, if $d = 4$, the covariance can be written as
\[
  \calK_{i, i'}^{(N)} = 2\inn{u_N^{(i)}}{u_N^{(i')}}^2 - \inn{(u_N^{(i)})^{\circ 2}}{(u_N^{(i')})^{\circ 2}},
\]
where $\cdot^{\circ p}$ denotes the $p$-th Hadamard product power.

In special cases, the covariance simplifies further. Of particular interest are when the contracting vectors are localized versus delocalized. A straightforward application of \eqref{eq:general_covariance} for $\sigma_\lambda^2 \equiv 1$ yields the following behavior (cf.\ Example \ref{eg:common_contracting_vector}):
\begin{align*}
  \norm{u_N^{(i)}}_\infty \to 1 &\implies \calK_{i, i'} = \lim_{N \to \infty} \inn{(u_N^{(i)})^{\circ d-2}}{(u_N^{(i')})^{\circ d-2}}; \\
  \norm{u_N^{(i)}}_\infty \to 0 &\implies \calK_{i, i'} = (d-2)!\left(\lim_{N \to \infty} \inn{u_N^{(i)}}{u_N^{(i')}}\right)^{d-2}; \\
  \text{supp}(u_N^{(i)}) \cap \text{supp}(u_N^{(i')}) = \emptyset &\implies \calK_{i, i'}^{(N)} = 0.
\end{align*}
Note that the first two cases do not make any assumptions on the shape of the contracting vector $u_N^{(i')}$ used to construct the covarying matrix $\frac{1}{\sqrt{N}}\mbf{T}_{d, N}[(u_N^{(i')})^{\otimes d-2}]$.

The variance profile $\sigma_\lambda^2 \equiv 1$ does admit an alternative formulation for the covariances $(\calK_{i, i'}^{(N)})_{i, i' \in I}$ that can be more tractable computationally. In particular, by considering $(\mcal{UBP}(d-2), \leq)$ as a poset for the reversed refinement order $\leq$, we can use the corresponding M\"{o}bius function $\mob_{\mcal{UBP}}$ to rewrite \eqref{eq:general_covariance} as\small
\[
    \mcal{K}_{i, i'}^{(N)} = \sum_{\rho \in \mcal{UBP}(d-2)} \Bigg(\sum_{\substack{\pi \in \mcal{UBP}(d-2) \\\text{s.t. } \rho \leq \pi}} \mob_{\mcal{UBP}}(\rho, \pi)\Bigg) \prod_{B \in \rho} \inn{\circ_{l \in B_L} u_N^{(i, l)}}{\circ_{r \in B_R} u_N^{(i', r)}}.
\]\normalsize
\end{example}

\subsubsection{General contractions}\label{sec:general_tensor_contractions}
We can also consider contractions by general unit tensors $(\mbf{U}_{d-2, N}^{(i)})_{i \in I} \subset \sphr^{N^{d-2} - 1}$. For the GOTE variance profile, the new formula for the covariance follows essentially from linearity:
\[
    \mcal{K}_{i ,i'}^{(N)} = \frac{1}{d(d-1)}\inn{\sym(\mbf{U}_{d-2, N}^{(i)})}{\sym(\mbf{U}_{d-2, N}^{(i')})}.
\]
One needs to take care since the contracting tensor $\mbf{U}_{d-2, N}^{(i)}$ may require an increasing number of basis elements to write as $N \to \infty$, which is why the result does not follow directly from linearity. The modification of the pictures in Section \ref{sec:graphs_of_tensors} and the corresponding calculations is straightforward but tedious. As before, the upper bound $\calK_{i, i}^{(N)} \leq \frac{1}{d(d-1)}$ for the variance is achieved iff $\mbf{U}_{d-2, N}^{(i)}$ is symmetric. On the other hand, the lower bound from Example \ref{eg:extreme_cases} no longer holds. Indeed, a natural candidate for maximally asymmetric in this context would be an antisymmetric tensor $\mbf{U}_{d-2, N}^{(i)}$, for which $\sym(\mbf{U}_{d-2, N}^{(i)}) = \mbf{0}_{d-2, N}$ is the zero tensor and $\ten[\mbf{U}_{d-2, N}^{(i)}] \equiv \mbf{0}_{2, N}$ is the zero matrix.

\subsubsection{Independent Wigner tensors}\label{sec:independent_tensors}
A key aspect of the contracted tensor ensemble is the singular source of randomness coming from the ambient tensor $\mbf{T}_{d, N}$. One can also consider contractions $(\mbf{U}_{d-2, N}^{(i, j)})_{i \in I, j \in J} \subset \sphr^{N^{d-2} - 1}$ of independent Wigner tensors $(\mbf{T}_{d, N}^{(j)})_{j \in J}$, in which case one obtains the expected answer based on the well-established principle of free independence emerging in the large $N$ limit of suitably generic independent random matrices. The augmented covariance matrix $(\calK_{i, i', j, j'}^{(N)})_{i, i' \in I, j, j' \in J}$ now takes the form
\[
    \calK_{i, i', j, j'}^{(N)} = \frac{\indc{j = j'}}{d(d-1)}\inn{\sym(\mbf{U}_{d-2, N}^{(i)})}{\sym(\mbf{U}_{d-2, N}^{(i')})}.
\]

\subsection{Comparison with existing results}\label{sec:existing_results}
Of course, one can simply view a contracted Wigner tensor as a generalization of a Wigner matrix where one allows for a dependence structure between the entries. Results in this direction naturally depend on the particulars of the dependence structure. To the best of our knowledge, nothing in the existing literature covers the case of a contracted Wigner tensor, even in the single-matrix model. For concreteness, consider a Wigner tensor $\mbf{T}_{4, N}$ with suitably scaled Rademacher entries under the common contracting vector
\begin{equation}\label{eq:contraction_example}
  u_N = \left(0, 0, \frac{1}{\sqrt{3}}, \frac{1}{\sqrt{3}}, \frac{1}{\sqrt{3(N-4)}}, \ldots, \frac{1}{\sqrt{3(N-4)}} \right).
\end{equation}
If one partitions $\mbf{T}_{4, N}[u_N^{\otimes 2}]$ so that entries across different blocks are independent, one necessarily obtains a block $B$ of size $\#(B) \sim N^2$, which rules out the partitioning scheme in \cite{SSB05}. Moreover, knowing that the entries other than the $(3, 3)$-th entry of $\mbf{T}_{4, N}[u_N^{\otimes 2}]$ have each attained their maximal value forces the remaining entry $\mbf{T}_{4, N}[u_N^{\otimes 2}](3, 3)$ to attain its maximal value as well, which rules out the conditional centeredness condition in \cite{GNT15}. Next, a straightforward calculation shows that
\begin{align*}
  \cov\big(\mbf{T}_{4, N}[u_N^{\otimes 2}](1, 1), \mbf{T}_{4, N}[u_N^{\otimes 2}](2, 1)\big) &= 0; \\
  \cov\big(\mbf{T}_{4, N}[u_N^{\otimes 2}](3, 3), \mbf{T}_{4, N}[u_N^{\otimes 2}](4, 3)\big) &= \frac{7}{54}, 
\end{align*}
which rules out the functional representation in \cite{BMP15} and the approximate uncorrelation scheme in \cite{HKW16}. Finally, a similar calculation shows that
\[
  \lim_{N \to \infty} \frac{1}{N} \E\left[\sum_{k \in [N]} \mbf{T}_{4, N}[u_N^{\otimes 2}](3, k)\mbf{T}_{4, N}[u_N^{\otimes 2}](4, k) \right] = \frac{1}{18},
\]
which rules out the isotropy condition in \cite{BBvH21}.

\section{Background}\label{sec:background}

\subsection{Graphs of tensors}\label{sec:graphs_of_tensors} 
Diagrammatic representations for tensors date back to the work of Penrose \cite{Pen71}. Roughly speaking, in Penrose graphical notation, tensors are represented by shapes such as boxes, circles, and triangles; the corresponding indices are notated by lines emanating from the shapes. If a line connects two shapes, then the corresponding index is summed over in both tensors.

In contrast, the RMT convention represents matrices by edges and indices by vertices. One can then keep track of repetitions in the indices by identifying vertices. This bookkeeping allows one to identify leading order terms when computing trace expansions. Such considerations form the basis of a noncommutative probability theory known as \emph{traffic probability} \cite{Mal20,CDM16,ACDGM21}. Of course, one can simply view the contracted tensor $\mbf{T}_{d, N}[u_N^{(i, 1)} \otimes \cdots \otimes u_N^{(i, d-2)}]$ as a random matrix and represent it in the usual way; however, the graph then fails to keep track of the dependencies between the entries introduced by the contractions.

We circumvent this issue by amalgamating the two representations. In particular, we work with a bipartite graph with a class of vertices $V$ for the indices and a class of vertices $W$ for the tensors. It will be convenient to further separate the tensors $W = W_1 \sqcup W_2$ into the $d$-th order tensors $W_1$ and the contracting vectors $W_2$. In our diagrams, we use circles for the indices, diamonds for the tensors, and boxes for the contracting vectors. For example,
  \[
    \begin{tikzpicture}[baseline=(current  bounding  box.center)]
    \node at (5, .0625) {$= \, \mbf{T}_{4,N}(k_1, k_2, l_1, l_2)u_N^{(1)}(l_1)u_N^{(2)}(l_2)$};
    
    \node[draw, diamond, inner sep=2.5pt] (T1) at (0,0) {$\mbf{T}_{4, N}$};
    \node[draw, rectangle] (U1) at (-1.06066017178, 2.12132034356) {$u_N^{(1)}$};
    \node[draw, rectangle] (U2) at (1.06066017178, 2.12132034356) {$u_N^{(2)}$};
    
    \node[draw, circle, inner sep=2pt] (I1) at (-1.5, 0) {$k_1$};
    \node[draw, circle, inner sep=2pt] (I2) at (-1.06066017178, 1.06066017178) {$l_1$};
    \node[draw, circle, inner sep=2pt] (I3) at (1.06066017178, 1.06066017178) {$l_2$};
    \node[draw, circle, inner sep=2pt] (I4) at (1.5, 0) {$k_2$};

    \draw[-] (T1) to (I1);
    \draw[-] (T1) to (I2);
    \draw[-] (T1) to (I3);
    \draw[-] (T1) to (I4);
    \draw[-] (U1) to (I2);
    \draw[-] (U2) to (I3);
    \end{tikzpicture}
  \]
  The symmetry of our tensor $\mbf{T}_{d, N}$ ensures that this representation is well-defined. We adopt the convention that any circular vertex without a label indicates a summation over all possible indices. For example,
  \[
    \begin{tikzpicture}[baseline=(current  bounding  box.center)]
    \node at (6, -.155) {$ = \displaystyle \sum_{l_1 \in [N]} \sum_{l_2 \in [N]} \mbf{T}_{4,N}(k_1, k_2, l_1, l_2)u_N^{(1)}(l_1)u_N^{(2)}(l_2)$};
    
    \node at (4.375,-1.28) {$ = \mbf{T}_{4, N}[u_N^{(1)} \otimes u_N^{(2)}](k_1, k_2)$};
    
    \node[draw, diamond, inner sep=2.5pt] (T1) at (0,0) {$\mbf{T}_{4, N}$};
    \node[draw, rectangle] (U1) at (-1.06066017178, 2.12132034356) {$u_N^{(1)}$};
    \node[draw, rectangle] (U2) at (1.06066017178, 2.12132034356) {$u_N^{(2)}$};
    
    \node[draw, circle, inner sep=2pt] (I1) at (-1.5, 0) {$k_1$};
    \node[draw, circle, inner sep=2pt] (I2) at (-1.06066017178, 1.06066017178) {$\phantom{l_1}$};
    \node[draw, circle, inner sep=2pt] (I3) at (1.06066017178, 1.06066017178) {$\phantom{l_2}$};
    \node[draw, circle, inner sep=2pt] (I4) at (1.5, 0) {$k_2$};
    
    \draw[-] (T1) to (I1);
    \draw[-] (T1) to (I2);
    \draw[-] (T1) to (I3);
    \draw[-] (T1) to (I4);
    \draw[-] (U1) to (I2);
    \draw[-] (U2) to (I3);
    \end{tikzpicture}
  \]
  The (unnormalized) trace of a general $m$-fold product
  \[
    \trace\big(\mbf{T}_{d, N}[u_N^{(i_1, 1)} \otimes \cdots \otimes u_N^{(i_1, d-2)}] \cdots \mbf{T}_{d, N}[u_N^{(i_m, 1)} \otimes \cdots \otimes u_N^{(i_m, d-2)}]\big)
  \]
  can then be represented as a decorated cycle\small
  \begin{equation}\label{eq:graph_of_tensors}
    \begin{tikzpicture}[baseline=(current  bounding  box.center)]    
    \node[draw, diamond, inner sep=2.5pt] (T1) at (0,2.59807621135) {$\mbf{T}_{d, N}$};
    \node[draw, rectangle] (U_1^1) at (-1.06066017178, 4.71939655491) {$u_N^{(i_1, 1)}$};
    \node[draw, rectangle] (U_1^{d-2}) at (1.06066017178, 4.71939655491) {$u_N^{(i_1, d-2)}$};

    \node[draw, diamond, inner sep=2.5pt] (T2) at (-2.25,1.29903810568) {$\mbf{T}_{d, N}$};
    \node[draw, rectangle] (U_2^1) at (-5.12571651460, 1.73458893811) {$u_N^{(i_2, 1)}$};
    \node[draw, rectangle] (U_2^{d-2}) at (-4.06505634282, 3.57170624519) {$u_N^{(i_2, d-2)}$};
        
    \node[draw, diamond, inner sep=2.5pt] (Tm) at (2.25,1.29903810568) {$\mbf{T}_{d, N}$};
    \node[draw, rectangle] (U_m^1) at (4.06505634282, 3.57170624519) {$u_N^{(i_m, 1)}$};
    \node[draw, rectangle] (U_m^{d-2}) at (5.12571651460, 1.73458893811) {$u_N^{(i_m, d-2)}$};
        
    \node[draw, diamond, inner sep=2.5pt] (T3) at (-2.25,-1.29903810568) {$\mbf{T}_{d, N}$};
    \node[draw, rectangle] (U_3^{d-2}) at (-5.12571651460, -1.73458893811) {$u_N^{(i_3, d-2)}$};
    \node[draw, rectangle] (U_3^1) at (-4.06505634282, -3.57170624519) {$u_N^{(i_3, 1)}$};
    
    \node[draw, diamond, inner sep=2.5pt] (T_{m-1}) at (2.25,-1.29903810568) {$\mbf{T}_{d, N}$};
    \node[draw, rectangle] (U_{m-1}^1) at (5.12571651460, -1.73458893811) {$u_N^{(i_{m-1}, 1)}$};
    \node[draw, rectangle] (U_{m-1}^{d-2}) at (4.06505634282, -3.57170624519) {$u_N^{(i_{m-1}, d-2)}$};
    
    \node[draw, circle, inner sep=2pt] (I_1^1) at (-1.5, 2.59807621135) {$\phantom{k}$};
    \node[draw, circle, inner sep=2pt] (I_1^2) at (-1.06066017178, 3.65873638313) {$\phantom{k}$};
    \node[draw, circle, inner sep=2pt] (I_1^3) at (1.06066017178, 3.65873638313) {$\phantom{k}$};
    \node[draw, circle, inner sep=2pt] (I_1^4) at (1.5, 2.59807621135) {$\phantom{k}$};
    \node at (0, 3.65873638313) {$\cdots$};

    \node[draw, circle, inner sep=2pt] (I_2^1) at (-3, 0) {$\phantom{k}$};
    \node[draw, circle, inner sep=2pt] (I_2^2) at (-3.69888873943, 0.9108095380) {$\phantom{k}$};
    \node[draw, circle, inner sep=2pt] (I_2^3) at (-2.63822856765, 2.74792684511) {$\phantom{k}$};
    \node at (-3.16855865354, 1.829368191566) {\rotatebox[origin=c]{60}{$\cdots$}};
        
    \node[draw, circle, inner sep=2pt] (I_m^1) at (3, 0) {$\phantom{k}$};
    \node[draw, circle, inner sep=2pt] (I_m^2) at (2.63822856765, 2.74792684511) {$\phantom{k}$};
    \node[draw, circle, inner sep=2pt] (I_m^3) at (3.69888873943, 0.91080953802) {$\phantom{k}$};
    \node at (3.16855865354, 1.829368191566) {\rotatebox[origin=c]{-60}{$\cdots$}};

    \node[draw, circle, inner sep=2pt] (I_3^1) at (-1.5, -2.59807621135) {$\phantom{k}$};
    \node[draw, circle, inner sep=2pt] (I_3^2) at (-2.63822856765, -2.74792684511) {$\phantom{k}$};
    \node[draw, circle, inner sep=2pt] (I_3^3) at (-3.69888873943, -0.9108095380) {$\phantom{k}$};
    \node at (-3.16855865354, -1.829368191566) {\rotatebox[origin=c]{-60}{$\cdots$}};
    
    \node[draw, circle, inner sep=2pt] (I_{m-1}^1) at (1.5, -2.59807621135) {$\phantom{k}$};
    \node[draw, circle, inner sep=2pt] (I_{m-1}^2) at (3.69888873943, -0.91080953802) {$\phantom{k}$};
    \node[draw, circle, inner sep=2pt] (I_{m-1}^3) at (2.63822856765, -2.74792684511) {$\phantom{k}$};
    \node at (3.16855865354, -1.829368191566) {\rotatebox[origin=c]{60}{$\cdots$}};
    
    \node at (0, -2.59807621135) {$\cdots$};
    
    \draw[-] (T1) to (I_1^1);
    \draw[-] (T1) to (I_1^2);
    \draw[-] (T1) to (I_1^3);
    \draw[-] (T1) to (I_1^4);
    
    \draw[-] (U_1^1) to (I_1^2);
    \draw[-] (U_1^{d-2}) to (I_1^3);
    
    \draw[-] (T2) to (I_1^1);
    \draw[-] (T2) to (I_2^1);
    \draw[-] (T2) to (I_2^2);
    \draw[-] (T2) to (I_2^3);

    \draw[-] (U_2^1) to (I_2^2);
    \draw[-] (U_2^{d-2}) to (I_2^3);

    \draw[-] (T3) to (I_2^1);
    \draw[-] (T3) to (I_3^1);
    \draw[-] (T3) to (I_3^2);
    \draw[-] (T3) to (I_3^3);

    \draw[-] (U_3^1) to (I_3^2);
    \draw[-] (U_3^{d-2}) to (I_3^3);
     
    \draw[-] (T_{m-1}) to (I_m^1);
    \draw[-] (T_{m-1}) to (I_{m-1}^1);
    \draw[-] (T_{m-1}) to (I_{m-1}^2);
    \draw[-] (T_{m-1}) to (I_{m-1}^3);

    \draw[-] (U_{m-1}^1) to (I_{m-1}^2);
    \draw[-] (U_{m-1}^{d-2}) to (I_{m-1}^3);
    
    \draw[-] (Tm) to (I_1^4);
    \draw[-] (Tm) to (I_m^1);
    \draw[-] (Tm) to (I_m^2);
    \draw[-] (Tm) to (I_m^3);

    \draw[-] (U_m^1) to (I_m^2);
    \draw[-] (U_m^{d-2}) to (I_m^3);

    \draw[-] (I_3^1) to (-.625, -2.59807621135);
    \draw[-] (I_{m-1}^1) to (.625, -2.59807621135);
    \end{tikzpicture}
  \end{equation}\normalsize

  Let $G = (V, W, E)$ be the graph above, where for convenience we suppress the dependence on $m$. Recall that $V$ is the set of (circular) index vertices and $W$ is the set of (polygonal) tensor vertices. We use $\ssim$ to denote adjacency in $G$. The expected normalized trace associated to the graph $G$ can then be written as
  \[
    \tau[G] := \frac{1}{N} \sum_{\phi: V \to [N]} \E\bigg[\prod_{w_1 \in W_1} \mbf{T}_{d, N}(w_1 | \phi)\bigg] \prod_{w_2 \in W_2} u_N(w_2 | \phi),
  \]
  where
  \begin{align*}
    \mbf{T}_{d, N}(w_1 | \phi) &:= \mbf{T}_{d, N}\Big(\bigtimes_{\substack{v \in V: \\ v\ssim w_1}} \phi(v)\Big); \\
    u_N(w_2 | \phi) &:=  u_N^{(w_2)}(\phi(v^{w_2})),
  \end{align*}
  and $v^{w_2}$ denotes the unique vertex in $V$ adjacent to $w_2.$ In other words, the notation $(w | \phi)$ is shorthand for the $\phi$-labels of the neighbors of $w$. By a slight abuse of notation, we assume that the superscript in $u_N^{(w_2)}$ identifies the contracting unit vector $u_N^{(i, j)}$.
  
  Similarly, we define
  \begin{equation}\label{eq:injective_traffic_state}
    \tau^0[G] := \frac{1}{N} \sum_{\phi: V \hookrightarrow [N]} \E\bigg[\prod_{w_1 \in W_1} \mbf{T}_{d, N}(w_1 | \phi)\bigg] \prod_{w_2 \in W_2} u_N(w_2 | \phi),
  \end{equation}
  where $\phi: V \hookrightarrow [N]$ denotes an injective map. Note that the functions $\tau$ and $\tau^0$ can be extended to arbitrary graphs of tensors in the obvious way. For any such graph, which we again denote by $G = (V, W, E)$, we write $\mcal{P}(V)$ for the set of partitions of the index vertices $V$. For a partition $\pi \in \mcal{P}(V)$, we define the \emph{quotient graph} $G^\pi$ as the graph obtained from $G$ by identifying vertices in the same block in $\pi$. Formally, $G^\pi = (V^\pi, W, E)$, where
  \[
    V^\pi := V/\ssim_\pi = \{B : B \in \pi\} = \{[v]_\pi : v \in V\}.
  \]
  This allows us to formulate the relation
  \begin{equation}\label{eq:traffic_mobius_relation}
    \tau[G] = \sum_{\pi \in \mcal{P}(V)} \tau^0[G^\pi].
  \end{equation}
  In the context of matrices, the function $\tau$ (resp., $\tau^0$) is known as the \emph{traffic state} (resp., \emph{injective traffic state}) \cite{Mal20}.

  To use equation \eqref{eq:traffic_mobius_relation}, we need to calculate $\tau^0$ on all possible quotients of $G$. The advantage of this comes from the fact that the injective labelings $\phi: V^\pi \hookrightarrow [N]$ will allow us to read off the dependence structure of the tensor entries directly from the graph $G^\pi$. To recognize the corresponding limit object, we use a characterization of semicircular families implicit in \cite{CDM16} (see also \cite[Section 3]{AM20}). We review this and other free probabilistic notions in the next section.

\subsection{Free probability}\label{sec:free_probability}

We content ourselves with a narrow review of the free probability machinery. The interested reader should consult \cite{VDN92,NS06,MS17} for further details.

\begin{definition}[Noncommutative probability]\label{defn:nc_probability}
  A \emph{noncommutative probability space (ncps)} is a unital algebra $\mcal{A}$ over $\C$ paired with a unital linear functional $\varphi: \mcal{A} \to \C$. For a family of random variables $\mbf{a} = (a_i)_{i \in I} \subset \mcal{A}$, their \emph{(joint) distribution} is defined as the linear functional
  \[
    \mu_{\mbf{a}}: \C\langle\mbf{x}\rangle \to \C, \qquad P \mapsto \varphi\big(P(\mbf{a})\big),
  \]
  where $\C\langle\mbf{x}\rangle$ is the free algebra over the indeterminates $\mbf{x} = (x_i)_{i \in I}$ and $P(\mbf{a}) \in \mcal{A}$ is the usual evaluation of noncommutative polynomials. We say that a sequence of families $\mbf{a}_N = (a_N^{(i)})_{i \in I} \subset (\mcal{A}_N, \varphi_N)$ \emph{converges in distribution} if the corresponding sequence of linear functionals $(\mu_{\mbf{a}_N})_{N \in \N}$ converges pointwise. Note that the limit defines a new ncps $(\C\langle\mbf{x}\rangle, \lim_{N \to \infty} \mu_{\mbf{a}_N})$.
\end{definition}

In other words, the distribution in this setting is the information of the noncommutative moments. Just as classical moments have a parameterization in terms of cumulants, noncommutative moments have a parameterization in terms of \emph{free cumulants}.

\begin{definition}[Free independence]\label{defn:free_independence}
  Let $\mcal{NC}(n)$ denote the set of noncrossing partitions of $[n]$. We write $\leq$ for the reversed refinement order on $\mcal{NC}(n)$ and $\mob_{\mcal{NC}}$ for the corresponding M\"{o}bius function. Note that the minimal element $0_n$ in $(\mcal{NC}(n), \leq)$ consists of singletons, whereas the maximal element $1_n$ is itself a singleton.

  A noncrossing partition $\pi \in \mcal{NC}(n)$ defines a multilinear functional
\[
  \varphi_\pi : \mcal{A}^n \to \C, \qquad (a_1, \ldots, a_n) \mapsto \prod_{B \in \pi} \varphi(B)[a_1, \ldots, a_n],
\] 
where a block $B = (i_1 < \cdots < i_m) \in \pi$ splits the $n$-tuple $(a_1, \ldots, a_n)$ into a lower order moment
\[
  \varphi(B)[a_1, \ldots, a_n] := \varphi(a_{i_1} \cdots a_{i_m}).
\]
The \emph{free cumulant} $\kappa_\pi : \mcal{A}^n \to \C$ is obtained from the M\"{o}bius convolution
\[
  \kappa_\pi[a_1, \ldots, a_n] := \sum_{\substack{\sigma \in \mcal{NC}(n) \\ \text{s.t. } \sigma \leq \pi}} \varphi_\sigma[a_1, \ldots, a_n] \mob_{\mcal{NC}}(\sigma, \pi).
\]
By the M\"{o}bius inversion formula, this is equivalent to
\begin{equation}\label{eq:free_cumulants_sum}
  \varphi_\pi[a_1, \ldots, a_n] = \sum_{\substack{\sigma \in \mcal{NC}(n) \\ \text{s.t. } \sigma \leq \pi}} \kappa_\sigma[a_1, \ldots, a_n].
\end{equation}
The free cumulants inherit both multilinearity and multiplicativity with respect to the blocks. In particular, 
\[
  \kappa_\pi[a_1, \ldots, a_n] = \prod_{B \in \pi} \kappa(B)[a_1, \ldots, a_n],
\] 
where $B = (i_1 < \cdots < i_m)$ is a block as before and
\[
  \kappa(B)[a_1, \ldots, a_n] := \kappa_{1_m}[a_{i_1}, \ldots, a_{i_m}].
\]
For notational convenience, we write $\kappa_m = \kappa_{1_m}$.

Subsets $(\mcal{S}_i)_{i \in I}$ of a ncps $(\mcal{A}, \varphi)$ are \emph{freely independent} if mixed free cumulants on the $(\mcal{S}_i)_{i \in I}$ vanish: for any $a_1, \ldots, a_n$ such that $a_j \in \mcal{S}_{i(j)}$,
\[
\exists i(j) \neq i(k) \implies \kappa_n[a_1, \ldots, a_n] = 0.
\]
A sequence of families $\mbf{a}_N \subset (\mcal{A}_N, \varphi_N)$ is \emph{asymptotically free} if it converges in distribution to a limit $\mbf{x}$ that is free in $(\C\langle\mbf{x}\rangle, \lim_{N \to \infty} \mu_{\mbf{a}_N})$.
\end{definition}

The noncommutative analogue of the multivariate normal distribution then follows from the noncrossing analogue of Wick's formula.

\begin{definition}[Semicircular family]\label{defn:semicircular_family}
  Let $(\mbf{K}_{i, i'})_{i, i' \in I}$ be a real positive semidefinite matrix. A \emph{(centered) semicircular family of covariance $\mbf{K}$} is a family of random variables $\mbf{s} = (s_i)_{i \in I}$ in a ncps $(\mcal{A}, \varphi)$ such that
\[
  \kappa_2(s_i, s_{i'}) = \mbf{K}_{i, i'}, \qquad \forall i, i' \in I,
\]
with all other cumulants on $\mbf{s}$ vanishing. By the moment-cumulant relation \eqref{eq:free_cumulants_sum}, this is equivalent to
\begin{equation}\label{eq:semicircular_family}
  \varphi(s_{i(1)} \cdots s_{i(n)}) = \sum_{\pi \in \mcal{NC}_2(n)} \prod_{\{j, k\} \in \pi} \mbf{K}_{i(j), i(k)},
\end{equation}
where $\mcal{NC}_2(n) \subset \mcal{NC}(n)$ is the subset of noncrossing \emph{pair} partitions of $[n]$. In particular, the $(s_i)_{i \in I}$ are free iff $(\mbf{K}_{i, i'})_{i, i' \in I}$ is diagonal, in which case we call $(s_i)_{i \in I}$ a \emph{free semicircular family}. If $\mbf{K}_{i, i'} = \indc{i = i'}$, then we call $(s_i)_{i \in I}$ a \emph{standard semicircular family}.

\end{definition}

\begin{example}\label{eg:rmt}
The algebra $L^{\infty-}(\Omega, \mcal{F}, \pr) \otimes \op{Mat}_N(\C)$ of random $N\times N$ matrices whose entries have finite moments of all orders equipped with the expected trace $\frac{1}{N}\E[\trace(\cdot)]$ defines a ncps. A well-known result of Dykema \cite[Theorem 2.1]{Dyk93}, building on earlier work of Voiculescu \cite{Voi91}, proves that independent Wigner matrices $(\mbf{W}_N^{(i)})_{i \in I} \subset (L^{\infty-}(\Omega, \mcal{F}, \pr) \otimes \op{Mat}_N(\C), \frac{1}{N}\E[\trace(\cdot)])$ converge in distribution to a standard semicircular family.
\end{example}

For random matrices $(\mbf{A}_N^{(i)})_{i \in I}$, one can also define convergence in distribution almost surely.

\begin{definition}\label{defn:as_convergence}
We say that $(\mbf{A}_N^{(i)})_{i \in I}$ \emph{converges in distribution almost surely} if for any finite subset $I_0 \subset I$, the sequence of random functionals
\[
  \mu_{(\mbf{A}_N^{(i)})_{i \in I_0}}: \C\langle\mbf{x}\rangle \times (\Omega, \mcal{F}, \pr) \to \C, \qquad P \mapsto \frac{1}{N}\mathrm{Tr}[P(\mbf{A}_N^{(i)} : i \in I_0)]
\]
converges pointwise almost surely.
\end{definition}

In special cases, one can read off the free cumulants from the injective traffic state. To see this, we recall the notion of a \emph{cactus graph}, a connected graph in which every edge belongs to a unique simple cycle. If the only quotients of a cycle that contribute in the injective traffic state are cactus graphs and the contributions are multiplicative with respect to the cycles of the cactus, then one can use a bijection between cactus graph quotients of a cycle of length $n$ and noncrossing partitions of $[n]$ to read off the free cumulants from the cycles of the cactus \cite[Section 5]{CDM16} (see also \cite[Section 3]{AM20}). A semicircular family can then be identified by the property that only cactus graphs with cycles of length two contribute (so-called \emph{double trees} \cite{Mal20}). We review this characterization in the next section. The picture for our contracted tensor ensemble is complicated by the additional vertices coming from the contracting vectors. Nevertheless, we show how the double tree criteria can still be adapted to this situation in Section \ref{sec:convergence_in_expectation}.

\subsection{The double tree criteria for \texorpdfstring{$d=2$}{d equals two}}\label{sec:double_trees}

We briefly review the $d = 2$ case to identify key properties that are sufficient to ensure convergence in distribution to a semicircular family. The techniques will serve as a guide for higher order $d$, where more sophisticated arguments are required to deal with the technical complications that arise when considering general tensors.

We start by noting that our picture for the trace of a general $m$-fold product in \eqref{eq:graph_of_tensors} applies equally well to the $d = 2$ case of random matrices $(\mbf{W}_N^{(i)})_{i \in I}$. Here, we no longer have the outer vertices decorating the cycle since there are $d - 2 = 0$ contractions. We can then simplify our graph $G$ by discarding the diamonds, which were only necessary to keep track of the additional indices in our $d$-th order tensor, and represent our $2$-tensor with a single labeled edge:
\[
    \begin{tikzpicture}[baseline=(current  bounding  box.center)]
    
    \node at (3, 0) {$=$};
    
    \node[draw, diamond, inner sep=1pt] (M1) at (0,0) {$\mbf{W}_{N}^{(i)}$};
    
    \node[draw, circle, inner sep=2pt] (I1) at (-1.5, 0) {$\phantom{k}$};
    \node[draw, circle, inner sep=2pt] (I2) at (1.5, 0) {$\phantom{k}$};
    
    \node[draw, circle, inner sep=2pt] (I3) at (4.5, 0) {$\phantom{k}$};
    \node[draw, circle, inner sep=2pt] (I4) at (7.5, 0) {$\phantom{k}$};

    \draw[-] (M1) to (I1);
    \draw[-] (M1) to (I2);
    \draw[-] (I3) to node[midway, above] {$i$} (I4);
    \end{tikzpicture}
\]
Let $H = (\msr{V}(H), \msr{E}(H))$ be the resulting graph: by construction, $H$ is a simple $m$-cycle in which each edge represents a matrix in $(\mbf{W}_N^{(i)})_{i\in I}$. In fact, the picture above shows how we can identify $(\msr{V}(H), \msr{E}(H)) = (V, W)$ with the vertices of our original graph $G = (V, W, E)$. We use $H$ to formulate two convenient properties that together guarantee convergence to a semicircular family.

\begin{proposition}\label{prop:casedequalto2}
Let $(\mbf{W}_N^{(i)})_{i \in I}$ be a family of random real symmetric matrices such that
\begin{enumerate}[label=(P\arabic*)]
    \item \label{cond:double_tree} (Supported on double trees)
    \[
      \lim_{N\to \infty} \tau^0[G^\pi]=0
    \]
    whenever $H^\pi$ is \emph{not} a double tree; 
    
    \item \label{cond:multiplicativity} (Multiplicativity) there is a real positive semidefinite matrix $(\mbf{K}_{i, i'})_{i, i' \in I}$ such that
    \[
      \lim_{N\to \infty} \tau^0[G^\pi] = \prod_{\{i, i'\}} \mbf{K}_{i, i'}
    \]
    whenever $H^\pi$ is a double tree, where the product runs over the double edges $\{i, i'\}$ in $H^\pi$.
\end{enumerate}
Then $(\mbf{W}_N^{(i)})_{i \in I}$ converges in distribution to a semicircular family $(s_i)_{i\in I}$ of covariance $\mbf{K}$.  
\end{proposition}

The proposition follows from a one-to-one correspondence between noncrossing partitions and cactus graph quotients of simple cycles \cite[Lemma 5.2]{CDM16} (see also \cite[Proposition 2.4]{AM20}). For the convenience of the reader, we outline this correspondence in the special case of double trees. We start by enumerating the vertices $\{v_1, \dots, v_m\}$ of the $m$-cycle $H$ in counterclockwise order. Similarly, we enumerate the edges $\{e_1, \ldots, e_m\}$ so that $v_i \stackrel{e_i}{\sim} v_{i+1}$
with the convention $v_{m+1}=v_1$. The map $[n]\ni i\mapsto v_i \in \msr{V}(H)$ defines a one-to-one correspondence between $\mathcal{NC}(n)$ and $\mathcal{NC}(\msr{V}(H))$. Similarly, for $[\bar{n}] = \{\bar{1}<\cdots < \bar{n}\}$, the map $[\bar{n}]\ni \bar{i} \mapsto e_i \in \msr{E}(H)$ defines a one-to-one correspondence between  $\mathcal{NC}(\bar{n})$ and  $\mathcal{NC}(\msr{E}(H))$. By considering the interlacing
\[
  1< \bar{1}< \cdots < n< \bar{n},
\]
we can transport the Kreweras complement \cite[Definition 9.21]{NS06}
\[
  \mathrm{Kr}: \mathcal{NC}(n)\to \mathcal{NC}(\bar{n})
\]
to a map 
\[
  \calK: \nc(\msr{V}(H))\to \nc(\msr{E}(H)).
\]
We will need the following lemma, which is a special case of \cite[Proposition 2.4]{AM20}. See Figure \ref{figure:double_tree} for an illustration.

\begin{lemma}[Double trees and noncrossing pairings]
\label{lem:kreweras}
The graph $H^\pi$ is a double tree iff $\pi\in \nc(\msr{V}(H))$ and $\calK(\pi)\in \nc_2(\msr{E}(H))$. Moreover, if $\calK(\pi)\in \nc_2(\msr{E}(H))$, then two edges are paired together iff their incident vertices are identified head-to-tail:
\[
  \{e_i, e_{i'}\} \in \mcal{K}(\pi) \iff [v_i]_\pi = [v_{i' + 1}]_\pi \quad \text{and} \quad [v_{i + 1}]_\pi = [v_{i'}]_\pi.
\]
\end{lemma}

\begin{figure}
\begingroup%
  \makeatletter%
  \providecommand\color[2][]{%
    \errmessage{(Inkscape) Color is used for the text in Inkscape, but the package 'color.sty' is not loaded}%
    \renewcommand\color[2][]{}%
  }%
  \providecommand\transparent[1]{%
    \errmessage{(Inkscape) Transparency is used (non-zero) for the text in Inkscape, but the package 'transparent.sty' is not loaded}%
    \renewcommand\transparent[1]{}%
  }%
  \providecommand\rotatebox[2]{#2}%
  \newcommand*\fsize{\dimexpr\f@size pt\relax}%
  \newcommand*\lineheight[1]{\fontsize{\fsize}{#1\fsize}\selectfont}%
  \ifx\svgwidth\undefined%
    \setlength{\unitlength}{360bp}%
    \ifx\svgscale\undefined%
      \relax%
    \else%
      \setlength{\unitlength}{\unitlength * \real{\svgscale}}%
    \fi%
  \else%
    \setlength{\unitlength}{\svgwidth}%
  \fi%
  \global\let\svgwidth\undefined%
  \global\let\svgscale\undefined%
  \makeatother%
  \begin{picture}(1,0.3)%
    \lineheight{1}%
    \setlength\tabcolsep{0pt}%
    \put(0.19862893,1.96292306){\color[rgb]{0,0,0}\makebox(0,0)[lt]{\begin{minipage}{0.14438897\unitlength}\raggedright \end{minipage}}}%
    \put(0,0){\includegraphics[width=\unitlength,page=1]{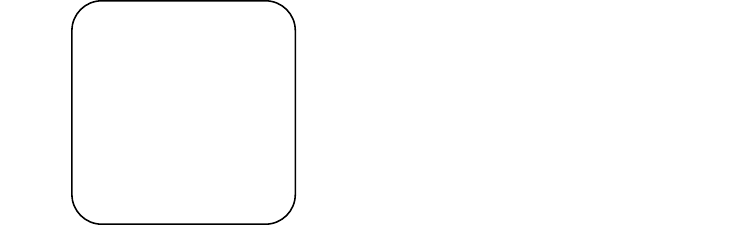}}%
    \put(0.48603101,0.14137883){\color[rgb]{0,0,0}\makebox(0,0)[lt]{\lineheight{1.66666663}\smash{\begin{tabular}[t]{l}$\stackrel{\pi}{\mapsto}$\end{tabular}}}}%
    \put(0,0){\includegraphics[width=\unitlength,page=2]{fig_double_tree.pdf}}%
  \end{picture}%
\endgroup%

\caption{An example of Lemma \ref{lem:kreweras} in action. On the left, the blocks of $\pi$ are indicated by the coloring of the vertices; the dotted lines connect the pairs in the corresponding Kreweras complement $\calK(\pi)$. The vertex identifications produce the quotient $H^\pi$ on the right, a double tree as guaranteed by the lemma.}\label{figure:double_tree}
\end{figure}

We can now give a short proof of Proposition \ref{prop:casedequalto2}.

\begin{proof}[Proof of Proposition \ref{prop:casedequalto2}]
If $G$ is the undecorated $m$-cycle in \eqref{eq:graph_of_tensors} with diamonds $\mbf{W}_N^{(i_1)}, \ldots, \mbf{W}_N^{(i_m)}$ in counterclockwise order, then
\[
  \tau[G] = \frac{1}{N} \E[ \trace(\mbf{W}_N^{(i_1)}\cdots
  \mbf{W}_N^{(i_m)})].
\]
Combining Lemma \ref{lem:kreweras} with properties \ref{cond:double_tree} and \ref{cond:multiplicativity}, we obtain
\begin{align*}
     \lim_{N\to \infty} \tau[G] &= \sum_{\pi\in \mathcal{P}(V)} \lim_{N\to \infty} \tau^0[G^\pi] \\
     &= \sum_{\substack{\pi\in \nc(\msr{V}(H)) \text{ s.t.}\\ \sigma = \calK(\pi) \in \mathcal{NC}_2(\msr{E}(H))}} \lim_{N\to \infty} \tau^0[G^\pi] \\
     &= \sum_{\sigma \in \mathcal{NC}_2(\msr{E}(H))} \prod_{\{i, i'\}\in \sigma} \mbf{K}_{i, i'},
\end{align*}
which is precisely the characterization of a semicircular family of covariance $\mbf{K}$ given in \eqref{eq:semicircular_family}.
\end{proof}

\section{Proofs of the main results}\label{sec:proof_of_theorem}

\subsection{Setup and complications for \texorpdfstring{$d \geq 3$}{d geq 3}}\label{sec:complications}

Let $G^\pi = (V^\pi, W, E)$ be a quotient of the graph $G$ in \eqref{eq:graph_of_tensors}. The analogue of \eqref{eq:injective_traffic_state} for the \emph{normalized} contracted tensor ensemble
\[
  \left(\frac{1}{\sqrt{N}}\mbf{T}_{d, N}\left[u_N^{(i, 1)} \otimes \cdots \otimes u_N^{(i, d-2)}\right]\right)_{i \in I}
\]
can be written as 
\[
  \tau^0[G^\pi] = \frac{1}{N^{\frac{m}{2} + 1}} \sum_{\phi: V^\pi \hookrightarrow [N]} \E\bigg[\prod_{w_1 \in W_1} \mbf{T}_{d, N}(w_1 | \phi)\bigg] 
  \prod_{w_2 \in W_2} u_N(w_2 | \phi).
\]
Since $|u_N^{(i, j)}(k)| \leq \norm{u_N^{(i, j)}}_2 = 1$, our strong moment assumption \eqref{eq:moment_bound} bounds
\begin{equation}\label{eq:bounded_summand}
  \E\bigg[\prod_{w_1 \in W_1} \mbf{T}_{d, N}(w_1 | \phi)\bigg]
  \prod_{w_2 \in W_2} u_N(w_2 | \phi) = O_m(1)
\end{equation}
uniformly in $(\pi, \phi)$. We immediately arrive at the crude asymptotic
\begin{equation}\label{eq:crude_asymptotic}
  \tau^0[G^\pi] =  O_m(N^{|V^\pi| - \frac{m}{2} - 1});
\end{equation}
however, in many cases, the independence and centeredness of our tensor entries result in 
\begin{equation}\label{eq:vanishing_summands}
  \E\bigg[\prod_{w_1 \in W_1} \mbf{T}_{d, N}(w_1 | \phi)\bigg] = 0, \qquad \forall \phi: V^\pi \hookrightarrow [N].
\end{equation}
To understand when this happens, we introduce some notation. For any $w_1 \in W_1$, we define $\mcal{B}_{\pi}(w_1)$ to be the set of neighbors of $w_1$ in $G^\pi$ with the additional information of edge multiplicity. Formally,
\[
  \mcal{B}_{\pi}(w_1) = \Big\{\big([v]_\pi, \#(\{e : [v]_\pi \stackrel{e}{\sim} w_1\})\big)\Big\} \subset V^\pi \times \N.
\]
In the special case of $\pi = \{\{v\} : v \in V\}$ and $G^\pi = G$, we simply write $\mcal{B}(w_1)$. By a slight abuse of notation, we sometimes refer to a vertex $[v]_\pi \in \mcal{B}_\pi(w_1)$ or a vertex $v \in \mcal{B}(w_1)$ instead of the tuple. For a map $\phi: V^\pi \hookrightarrow [N]$, the tensor entries $\mbf{T}_{d, N}(w_1 | \phi)$ and $\mbf{T}_{d, N}(w_1' | \phi)$ are dependent (in fact, identical) iff $B_\pi(w_1) = B_\pi(w_1')$, in which case we say that the diamonds $w_1$ and $w_1'$ are \emph{overlaid} in $G^\pi$.

Definition \ref{defn:wigner_tensor} \ref{cond:centered} of our Wigner tensor ensemble implies that \eqref{eq:vanishing_summands} holds unless
\begin{equation}\label{eq:centering_in_graph}
  \#(\mcal{B}_\pi(w_1)) = 1 \text{ or } \exists w_1' \in W_1\setminus\{w_1\}: \mcal{B}_\pi(w_1) = \mcal{B}_\pi(w_1'), \qquad \forall w_1 \in W_1.
\end{equation}
So, we can restrict ourselves to partitions $\pi \in \mcal{P}(V)$ that satisfy this property. Assuming the suboptimality of $\#(\mcal{B}_\pi(w_1)) = 1$ for now, this says that diamond overlays are common. To keep track of these, we define the function
\[
  \sigma_{(\cdot)}: \calP(V) \to \calP(W_1), \quad \pi \mapsto \sigma_\pi,
\]
where the partition $\sigma_\pi$ satisfies
\begin{equation}\label{eq:neighborhood_partition}
  w_1 \sim_{\sigma_\pi} w_1' \iff \mcal{B}_\pi(w_1) = \mcal{B}_\pi(w_1').
\end{equation} 

It will be necessary to classify the index vertices $V = V_{\op{in}} \sqcup V_{\op{out}}$ in $G$ according to their positions in the graph: the \emph{inner} vertices $V_{\op{in}}$ are those inside of the cycle (adjacent only to diamonds); the outer vertices $V_{\op{out}}$ are those outside of the cycle (adjacent to boxes). We also use an auxiliary graph $H^\pi$ to compress the information of $G^\pi$. 

\begin{definition}\label{defn:H_graph}[$H$ and $H^\pi$]
Let $H=(\msr{V}(H), \msr{E}(H))$ be the graph with vertices $\msr{V}(H) = \vin$ and edges $\msr{E}(H) = W_1$. Two vertices $v, v' \in \msr{V}(H)$ are connected by an edge $e \in \msr{E}(H)$ iff the diamond $e$ is adjacent to $v$ and $v'$ in $G$. For a partition $\pi \in \calP(V)$, we use the shorthand notation $H^\pi$ for the quotient graph $H^{\pi\restriction_{\msr{V}(H)}}$. 
\end{definition}

As before, $H$ is a simple $m$-cycle whose vertices are the inner vertices of $G$ and whose edges correspond to the diamonds of $G$. The construction of $H^\pi$ does not materially differ from the matrix case of $d=2$; however, one no longer has a simple characterization of the quotients that produce nonzero terms $\tau^0[G^\pi] \neq 0$. In particular, if $d \geq 3$, then outer vertices can be identified with inner vertices to create a diamond overlay while keeping the inner vertices separate in $H^\pi$. Consequently, edges do not even have to be incident in $H^\pi$ to be in the same block of $\sigma_\pi$. See Figure \ref{figure:overlay_square} and Figure \ref{figure:overlay_hexagon} for examples.

\begin{figure}
\begingroup%
  \makeatletter%
  \providecommand\color[2][]{%
    \errmessage{(Inkscape) Color is used for the text in Inkscape, but the package 'color.sty' is not loaded}%
    \renewcommand\color[2][]{}%
  }%
  \providecommand\transparent[1]{%
    \errmessage{(Inkscape) Transparency is used (non-zero) for the text in Inkscape, but the package 'transparent.sty' is not loaded}%
    \renewcommand\transparent[1]{}%
  }%
  \providecommand\rotatebox[2]{#2}%
  \newcommand*\fsize{\dimexpr\f@size pt\relax}%
  \newcommand*\lineheight[1]{\fontsize{\fsize}{#1\fsize}\selectfont}%
  \ifx\svgwidth\undefined%
    \setlength{\unitlength}{360bp}%
    \ifx\svgscale\undefined%
      \relax%
    \else%
      \setlength{\unitlength}{\unitlength * \real{\svgscale}}%
    \fi%
  \else%
    \setlength{\unitlength}{\svgwidth}%
  \fi%
  \global\let\svgwidth\undefined%
  \global\let\svgscale\undefined%
  \makeatother%
  \begin{picture}(1,0.35)%
    \lineheight{1}%
    \setlength\tabcolsep{0pt}%
    \put(0,0){\includegraphics[width=\unitlength,page=1]{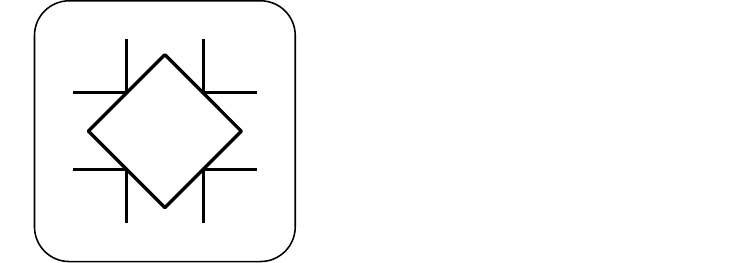}}%
    \put(0.19862893,1.96292303){\color[rgb]{0,0,0}\makebox(0,0)[lt]{\begin{minipage}{0.14438897\unitlength}\raggedright \end{minipage}}}%
    \put(0,0){\includegraphics[width=\unitlength,page=2]{fig_overlay_square.pdf}}%
    \put(0.48613873,0.16637829){\color[rgb]{0,0,0}\makebox(0,0)[lt]{\lineheight{1.66666675}\smash{\begin{tabular}[t]{l}$\stackrel{\pi}{\mapsto}$\end{tabular}}}}%
  \end{picture}%
\endgroup%

\caption{An example of a partition $\pi$ that overlays every diamond with at least one other diamond. For convenience, we have omitted the boxes corresponding to the contracting vectors. On the left, the blocks of $\pi$ are indicated by the coloring of the vertices. On the right, the dotted lines indicate the blocks of $\sigma_\pi$ in $H^\pi$. Note that the quotient graph $H^\pi$ is the same as the original graph $H$ in this case and that $\sigma_\pi$ is a \emph{crossing} pairing.}\label{figure:overlay_square}
\end{figure}

\begin{figure}
\begingroup%
  \makeatletter%
  \providecommand\color[2][]{%
    \errmessage{(Inkscape) Color is used for the text in Inkscape, but the package 'color.sty' is not loaded}%
    \renewcommand\color[2][]{}%
  }%
  \providecommand\transparent[1]{%
    \errmessage{(Inkscape) Transparency is used (non-zero) for the text in Inkscape, but the package 'transparent.sty' is not loaded}%
    \renewcommand\transparent[1]{}%
  }%
  \providecommand\rotatebox[2]{#2}%
  \newcommand*\fsize{\dimexpr\f@size pt\relax}%
  \newcommand*\lineheight[1]{\fontsize{\fsize}{#1\fsize}\selectfont}%
  \ifx\svgwidth\undefined%
    \setlength{\unitlength}{360bp}%
    \ifx\svgscale\undefined%
      \relax%
    \else%
      \setlength{\unitlength}{\unitlength * \real{\svgscale}}%
    \fi%
  \else%
    \setlength{\unitlength}{\svgwidth}%
  \fi%
  \global\let\svgwidth\undefined%
  \global\let\svgscale\undefined%
  \makeatother%
  \begin{picture}(1,0.35)%
    \lineheight{1}%
    \setlength\tabcolsep{0pt}%
    \put(0.19862893,1.96292308){\color[rgb]{0,0,0}\makebox(0,0)[lt]{\begin{minipage}{0.14438897\unitlength}\raggedright \end{minipage}}}%
    \put(0,0){\includegraphics[width=\unitlength,page=1]{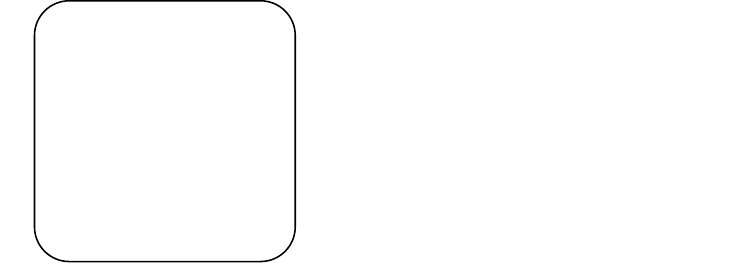}}%
    \put(0.48613874,0.16637815){\color[rgb]{0,0,0}\makebox(0,0)[lt]{\lineheight{1.66666663}\smash{\begin{tabular}[t]{l}$\stackrel{\pi}{\mapsto}$\end{tabular}}}}%
    \put(0,0){\includegraphics[width=\unitlength,page=2]{fig_overlay_hexagon.pdf}}%
  \end{picture}%
\endgroup%

\caption{Another example of a partition $\pi$ that overlays every diamond with at least one other diamond. Here, the yellow, pink, and purple vertices are identified in a way that resembles the optimal situation from the usual Wigner case. We will see that this is still not enough to overcome the defects in the other parts of the graph.}\label{figure:overlay_hexagon}
\end{figure}

This already rules out applying our crude asymptotic \eqref{eq:crude_asymptotic}, which is sufficient for the analysis in the $d = 2$ case. Roughly speaking, the diamond overlay condition allows for too many vertices in $G^\pi$: one can find partitions $\pi \in \mcal{P}(V)$ such that $\tau^0[G^\pi] \neq 0$ with
\[
  \#(\{\phi: V^\pi \hookrightarrow [N]\}) = \Theta(N^{(d-1)\frac{m}{2} + 1}).
\]
Since $d \geq 3$, this is much larger than the normalization term $N^{\frac{m}{2} + 1}$.

The issue comes from the overly generous bound of the summands in \eqref{eq:bounded_summand}, which can be of much smaller order depending on the vertex labels $\phi$. For example, if the contracting vectors are standard basis vectors, then the product $\prod_{w_2\in W_2} \un{w_2}{\phi}$
will be $0$ for most $\phi$. The opposite extreme would be if the contracting vectors were completely flat with $N^{-\frac{1}{2}}$ for every entry, in which case the product is $N^{-(d-2)\frac{m}{2}}$ for any $\phi$. Treating the general case without knowing the exact shape of the contracting vectors greatly complicates the analysis. Recognizing the importance of the inner-outer interactions, we split our partition $\pi = \pi_{\op{in}} \sqcup \pi_{\op{out}} \sqcup \pi_{\op{mix}}$ into
\begin{itemize}
    \item $\pi_{\op{in}}$, the set of blocks consisting only of inner vertices;
    \item $\pi_{\op{out}}$, the set of blocks consisting only of outer vertices;
    \item $\pi_{\op{mix}}$, the set of remaining blocks, each of which has at least one inner vertex and one outer vertex.
\end{itemize}
Our goal is then to understand how the number of outer vertices in each block of $\pi$ relates to the structure of $\sigma_\pi$. 

\subsection{Convergence in expectation}\label{sec:convergence_in_expectation}

In the previous section, we saw that a major complication comes from the existence of partitions $\pi\in \calP(V)$ such that $\tau^0[G^\pi]\neq 0$ and $\pim \neq \emptyset$. If we assume for the time being that such mixed partitions do not contribute in the limit, then we can prove convergence to a semicircular family in a manner similar to Proposition \ref{prop:casedequalto2}. Actually proving that mixed partitions do not contribute in the limit involves a delicate technical argument that we defer to Section \ref{sec:tech_sec_one}. 

Properties \ref{cond:double_tree} and \ref{cond:multiplicativity} in Proposition \ref{prop:casedequalto2} have natural extensions to the general setting of $d\geq 2$. Since we will need to show that the contracted ensemble satisfies these properties, we state them as lemmas. 

\begin{lemma}[Supported on double trees]
\label{lem:supportontrees}
\[
  \tau^0[G^\pi]=O_{m}(N^{-1})
\]
unless $H^\pi$ is a double tree and $\pi_{\op{mix}}=\emptyset$, where the asymptotic is uniform over all possible contracting vectors $(u_N^{(i, j)})_{i\in I, j \in [d-2]} \subset \sphr^{N-1}$.
\end{lemma}

\begin{lemma}[Multiplicativity]
\label{lem:multiplicativity}
For any partition $\rho\in \calP(\vin)$ such that $H^\rho$ is a double tree,
    \begin{equation}
    \label{eq:factorizationofcovariances}
     \sum_{\substack{\pi\in \calP(V) :\\  \pi\restriction_{\vin}=\rho,\, \pim=\emptyset}} \tau^0[G^\pi] = \prod_{\{i, i'\}} \mbf{K}_{i, i'}^{(N)}+O_{m, d}(N^{-1}),
    \end{equation}
     where the product runs over the double edges $\{i, i'\}$ in the double tree $H^\rho$ and $(\mbf{K}_{i, i'}^{(N)})_{i, i' \in I}$ is the Gram matrix defined in \eqref{eq:gram_matrix}.  As before, the asymptotic is uniform over all possible contracting vectors $(u_N^{(i, j)})_{i\in I, j \in [d-2]} \subset \sphr^{N-1}$.
\end{lemma}

The proof of Lemma \ref{lem:supportontrees} (resp., Lemma \ref{lem:multiplicativity}) is deferred to Section \ref{sec:tech_sec_one} (resp., the end of this section). Drawing inspiration from Proposition \ref{prop:casedequalto2}, we will now show that Lemmas \ref{lem:supportontrees} and \ref{lem:multiplicativity} imply that the mixed moments of the contracted tensor ensemble are close to the mixed moments of a semicircular family. We recall the notation
\begin{align*}
  P_N(i_1, \dots, i_m) = &\bigg(\frac{1}{\sqrt{N}} \mbf{T}_{d, N}\left[u_N^{(i_1, 1)} \otimes \cdots \otimes u_N^{(i_1, d-2)}\right]\bigg) \\
  &\cdots \bigg(\frac{1}{\sqrt{N}} \mbf{T}_{d, N}\left[u_N^{(i_m, 1)} \otimes \cdots \otimes u_N^{(i_m, d-2)}\right]\bigg),
\end{align*}
which allows us to state
\begin{proposition}
\label{prop:convergenceinexpectation}
 For any $i_1, \ldots, i_m\in I$,
 \[
   \bigg|\frac{1}{N} \E\Big[\trace\big( P_N(i_1, \dots, i_m)\big)\Big]- \varphi\big(s_{i_1}^{(N)}\cdots s_{i_m}^{(N)}\big)\bigg|= O_{m, d}(N^{-1}),
 \]
 where $\big(s_i^{(N)}\big)_{i \in I}$ is a semicircular family of covariance $(\mbf{K}_{i, i'}^{(N)})_{i, i' \in I}$ and the asymptotic is uniform over all possible contracting vectors $(u_N^{(i, j)})_{i\in I, j \in [d-2]} \subset \sphr^{N-1}$.
\end{proposition}

\begin{proof}[Proof assuming Lemmas \ref{lem:supportontrees} and  \ref{lem:multiplicativity}]
Recall that the graph $G$ was defined so that
\[
  \tau[G] = \frac{1}{N} \E\Big[ \trace\big(P_N(i_1, \dots, i_m)\big)\Big].
\]
By Lemma \ref{lem:kreweras}, if $\rho \in \calP(\vin)= \calP(\msr{V}(H))$, then $H^\rho$ is a double tree iff $\rho \in \nc(\vin)$ and $\calK(\rho)\in \mathcal{NC}_2(\msr{E}(H))$. Combining this with Lemma \ref{lem:supportontrees}, we obtain
\[
  \tau[G] =  \sum_{\pi \in \calP(V)}  \tau^0[G^\pi] = \sum_{\substack{\rho \in \nc(\vin): \\ \calK(\rho) \in \nc_2(\msr{E}(H)) }}   \sum_{\substack{\pi \in \calP(V) :\\ \pi\restriction_{\vin} =\rho, \pim=\emptyset}} \tau^0[G^\pi] + O_{m, d}(N^{-1}).
\]
Applying Lemma \ref{lem:multiplicativity} to the inner sum yields
\begin{align*}
  \tau[G] &= \sum_{\substack{\rho \in \nc(\vin): \\ \calK(\rho) \in \nc_2(\msr{E}(H)) }} \prod_{\{i, i'\}\in \calK(\rho)} \mbf{K}_{i, i'}^{(N)}+O_{m, d}(N^{-1}) \\
  &= \sum_{\sigma \in \nc_2(\msr{E}(H))} \prod_{\{i, i'\}\in \sigma} \mbf{K}_{i, i'}^{(N)}+ O_{m, d}(N^{-1}).
\end{align*}
Finally, since $\bK^{(N)}$ is positive semidefinite, the semicircular family $(s_{i}^{(N)})_{i\in I}$ is well-defined with
\[
   \varphi\big(s_{i_1}^{(N)}\cdots s_{i_m}^{(N)}\big) = \sum_{\sigma \in \nc_2(\msr{E}(H))} \prod_{\{i, i'\}\in \sigma} \mbf{K}_{i, i'}^{(N)}
\]
by the characterization in \eqref{eq:semicircular_family}. We conclude that $\big|\tau[G]-\varphi\big(s_{i_1}^{(N)}\cdots s_{i_m}^{(N)}\big)\big| = O_{m, d}(N^{-1})$.  
\end{proof}

Naturally, if the limits $\bK_{i, i'} = \lim_{N\to \infty} \bK_{i, i'}^{(N)}$ exist for every $i, i'\in I$, then Proposition \ref{prop:convergenceinexpectation} implies that the contracted tensor ensemble converges in distribution to a semicircular family of covariance $(\bK_{i, i'})_{i, i' \in I}$. We conclude this section by proving Lemma \ref{lem:multiplicativity}.

\begin{proof}[Proof of Lemma \ref{lem:multiplicativity}]
Let $\rho \in \calP(\vin)$ be such that $H^\rho$ is a double tree. We divide the proof into three steps. First, we identify the partitions $\pi \in \mcal{P}(V)$ that contribute to the sum in \eqref{eq:factorizationofcovariances} and use this to simplify the formula for each of the corresponding $\tau^0[G^\pi]$. Secondly, we show that the values of $\phi|_{\vin}$ can be ignored in these formulas. Finally, we prove that our simplified formulation of the sum in \eqref{eq:factorizationofcovariances} can be factored, up to a $O_{m, d}(N^{-1})$ correction, as a product of sums over uniform block permutations, which in turn can be rewritten using the advertised inner product in \eqref{eq:gram_matrix}.  

\subsubsection*{Step 1: Simplifying the formula for $\tau^0[G^\pi]$}
By Lemma \ref{lem:kreweras}, we know that $H^\rho$ is a double tree iff $\rho \in \mcal{NC}(V_{\op{in}})$ and $\mcal{K}(\rho) \in \mcal{NC}_2(\msr{E}(H)) = \mcal{NC}_2(W_1)$, where $\mcal{K}(\rho)$ is the pair partition whose pairs are precisely the double edges in $H^\rho$. The appearance of uniform block permutations (Definition \ref{defn:uniform_block_permutation}) in our formulas is a result of the following relabeling of the vertices in $G$ according to $\mcal{K}(\rho)$ as suggested by Figure \ref{figure:covariance}.

For each pair of diamonds $D \in \mcal{K}(\rho)$, we designate one as the \emph{left} diamond $w_1^{D_L}$ and the other as the \emph{right} diamond $w_1^{D_R}$ so that $D = \{w_1^{D_L}, w_1^{D_R}\}$. The outer vertices $V_{\op{out}}$ in $G$ can then be partitioned according to their neighboring diamond:
\[
  V_{\op{out}} = \bigsqcup_{D \in \mcal{K}(\rho)} [d-2]_{D} =\bigsqcup_{D \in \mcal{K}(\rho)} \Big([d-2]_{D_L} \sqcup [d-2]_{D_R}\Big),
\]
where
\begin{align*}
  [d-2]_{D_L} &= \{ v\in \vout: v\sim w_1^{D_L}\}; \\
  [d-2]_{D_R} &= \{ v\in \vout: v\sim w_1^{D_R}\}.
\end{align*}

\begin{figure}
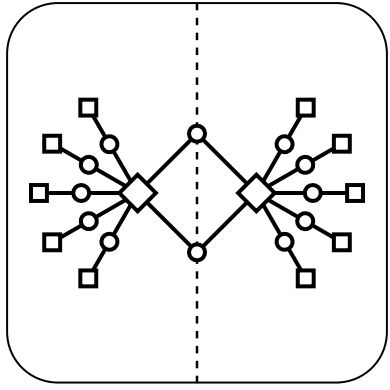
\caption{An example of the graph $G$ in \eqref{eq:graph_of_tensors} for $m = 2$. The dotted line down the middle is not part of the graph. We add it to motivate the appearance of uniform block permutations in our formula for the covariance. See the proof of Lemma \ref{lem:multiplicativity} for the precise details.}\label{figure:covariance}
\end{figure}

Now, suppose that $\pi \in \calP(V)$ satisfies $\pi\restriction_{\vin}=\rho$ and $\pim=\emptyset$. We can then factor $\pi=\pin \sqcup \pout$ into $\pin=\rho$ and $\pout\in \calP(\vout)$. Since $H^\rho$ is a double tree, every diamond in $G^\pi$ is adjacent to at least two index vertices $[v]_\pi \neq [v']_\pi$. In that case, the centeredness of the off-diagonal tensor entries in Definition \ref{defn:wigner_tensor} \ref{cond:centered} implies that $\tau^0[G^\pi] = 0$ unless every diamond is overlaid with at least one other diamond in $G^\pi$. Since $\pi_{\op{mix}} = \emptyset$ and $H^\rho$ is a double tree,
diamonds $w_1, w_1' \in W_1$ are overlaid in $G^\pi$ only if $\{w_1, w_1'\} \in \mcal{K}(\rho)$. It follows that $\sigma_\pi = \mcal{K}(\rho)$, where we recall that $\sigma_{(\cdot)}$ is the function that takes a partition $\pi \in \mcal{P}(V)$ to the corresponding partition of the diamonds $\sigma_\pi \in \mcal{P}(W_1)$ satisfying \eqref{eq:neighborhood_partition}. We can then restrict our sum
\[
  \sum_{\substack{\pi\in \calP(V) :\\  \pi\restriction_{\vin}=\rho,\, \pim=\emptyset}} \tau^0[G^\pi] = \sum_{\pi \in \mcal{P}_\rho} \tau^0[G^\pi]
\]
to the class of partitions 
\[
  \calP_\rho = \{\pi \in \calP(V) : \pin = \rho, \pim=\emptyset, \sigma_\pi = \mcal{K}(\rho)\}.
\]

If $\pi \in \mcal{P}_\rho$, then any diamond in $G^\pi$ will have at least three neighbors: two neighbors come from the fact that $H^\rho$ is a double tree, and at least one outer vertex neighbor survives from $G$ as a result of the non-mixing $\pim = \emptyset$. Since $\sigma_\pi = \mcal{K}(\rho)$, the independence assumption together with the variance profile in Definition \ref{defn:wigner_tensor} \ref{cond:gote_variance} of our tensor ensemble imply that for any $\phi : V^\pi \into [N]$,
\begin{align}
  \E\bigg[\prod_{w_1 \in W_1} \tdn{w_1}{\phi}\bigg] &= \prod_{D\in \mcal{K}(\rho)} \E\bigg[ \tdn{w_1^{D_L}}{\phi}\tdn{w_1^{D_R}}{\phi}\bigg] \nonumber \\
  \label{eq:expectationofproduct} &= \prod_{D\in \calK(\rho)} \binom{d}{b_1^{D, \phi}, \dots, b_{N}^{D, \phi}}^{-1},
\end{align}
where $b_l^{D, \phi}$ is the multiplicity of $l \in [N]$ as an index in $\ten(w_1^{D_L}|\phi)$ (equivalently, in $\ten(w_1^{D_R}|\phi)$). Note that while $(b_l^{D, \phi})_{l \in [N]}$ depends on $\phi$, the multiset of multiplicities $\{b_1^{D, \phi}, \dots, b_N^{D, \phi}\}$ only depends on the block structure of $\pi$ restricted to $[d-2]_{D}$. To keep track of this, we define the partition $\pout^D := \pout \restriction_{[d-2]_D}$. Since $w^{D_L}_1$ and $w_1^{D_R}$ are overlaid, we know that $\#(B\cap [d-2]_{D_L}) =\#(B\cap [d-2]_{D_R})$ for every $B \in \pout$. In particular, this implies that $\pout^D$ is a uniform block permutation of $[d-2]_D$. This will prove useful later, but for now we simply note that
\[
    \prod_{D \in \calK(\rho)} \binom{d}{b_1^{D, \phi}, \dots, b_{N}^{D, \phi}}^{-1} = \prod_{D\in \calK(\rho)} \frac{1}{d!} \prod_{B\in \pout^D} \left(\frac{\#(B)}{2}\right)!=: \nu_\pi
\]
only depends the choice of $\pi \in \mcal{P}_\rho$ and not the actual vertex labels $\phi: V^\pi \hookrightarrow [N]$. Combining this with \eqref{eq:expectationofproduct}, we can simplify the formula for $\tau^0[G^\pi]$ to
\begin{align}\label{eq:formulafortauzero}
    \tau^0[G^\pi] &= \frac{1}{N^{\frac{m}{2} + 1}} \sum_{\phi: \pi \hookrightarrow [N]} \E\bigg[\prod_{w_1 \in W_1} \tdn{w_1}{\phi}\bigg] 
  \prod_{w_2\in W_2} \un{w_2}{\phi} \nonumber \\ 
    &= \frac{\nu_\pi}{N^{\frac{m}{2} + 1}}   \sum_{\phi: \pi \hookrightarrow [N]}  
  \prod_{w_2\in W_2} \un{w_2}{\phi}. 
\end{align}

\subsubsection*{Step 2: Discarding the inner labels $\phi|_{\vin}$} Note that the terms $\prod_{w_2\in W_2} \un{w_2}{\phi}$ appearing in the sum above do not depend on the values of $\phi|_{\vin}$. Thus, for any $\pi \in \calP_\rho$,
\begin{align*}
    &\frac{1}{N^{\frac{m}{2} + 1}} \sum_{\phi: \pi \hookrightarrow [N]} \prod_{w_2\in W_2} \un{w_2}{\phi} \\
    = &\frac{1}{N^{\frac{m}{2} + 1}} \sum_{\phi: \pout \into [N]}\, \sum_{\substack{\phi': \pin \into [N] \text{ s.t.} \\ \phi([N])\cap \phi'([N])=\emptyset}}\, \prod_{w_2\in W_2} \un{w_2}{\phi} \\
    = &\frac{(N- \#(\pout))_{\frac{m}{2}+1}}{N^{\frac{m}{2}+1}} \sum_{\phi: \pout \into [N]} \prod_{w_2\in W_2} \un{w_2}{\phi},  
\end{align*}
where
\[
  (N - \#(\pout))_{\frac{m}{2}+1} = (N - \#(\pout)) (N - \#(\pout)-1) \cdots (N - \#(\pout)-\frac{m}{2})
\] 
is the falling factorial and we have used the fact that $\#(\pin) = \frac{m}{2}+1$ since $H^\rho$ is a double tree.

At the same time, since $\sigma_\pi = \mcal{K}(\rho)$ and $\pim = \emptyset$, we know that every block $B \in \pout$ is adjacent to at least two boxes in $G^\pi$. Combining the elementary bound $|u_N^{(i, j)}(k)| \leq \norm{u_N^{(i, j)}}_2 = 1$ with the Cauchy-Schwarz inequality, we see that
\begin{align*}
  \Big|\sum_{\phi: \pout \into [N]} \prod_{w_2\in W_2} \un{w_2}{ \phi}\Big| &\leq  \sum_{\phi: \pout \to [N]} \prod_{w_2\in W_2} |\un{w_2}{ \phi}|\\
  &= \prod_{B \in \pout} \Big(\sum_{k \in [N]} \prod_{\substack{w_2 \in W_2: \\ w_2 \sim_{G^\pi} B}} |u_N^{(w_2)}(k)|\Big) \leq 1.
\end{align*}
Since $(N - \#(\pout))_{\frac{m}{2}+1} = N^{\frac{m}{2}+1} + O_{m, d}(N^{\frac{m}{2}})$, this allows us to further simplify \eqref{eq:formulafortauzero} to
\[
  \tau^0[G^\pi] = \nu_\pi \sum_{\phi: \pout \into [N]} \prod_{w_2\in W_2} \un{w_2}{\phi} + O_{m, d}(N^{-1}).
\]

\subsubsection*{Step 3: Factoring and reassembling.} Putting everything together, we obtain\small
\begin{align}
  \sum_{\substack{\pi\in \calP(V) :\\  \pi\restriction_{\vin}=\rho,\, \pim=\emptyset }} \tau^0[G^\pi] &= \sum_{\pi \in \calP_\rho} \tau^0[G^\pi] \nonumber \\
    &= \sum_{\pi \in \calP_\rho} \nu_\pi \sum_{\phi: \pout \into [N]} \prod_{w_2\in W_2} \un{w_2}{ \phi} +O_{m, d}(N^{-1}) \nonumber \\
    \label{eq:tensor_wick} &= \sum_{\pi \in \calP_\rho} \sum_{\phi: \pout \into [N]} \prod_{D\in \mcal{K}(\rho)} \frac{1}{d!} \prod_{B\in \pout^D} \left(\frac{\#(B)}{2}\right)! \prod_{v\in B} u_N^{(w_2^v)}(\phi(B)) + O_{m, d}(N^{-1}),
\end{align}\normalsize
where $w_2^v$ denotes the unique box in $W_2$ that is adjacent to the outer vertex $v \in V_{\op{out}}$. Using the fact that $\pout^D \in \ubp([d-2]_D)$, we can factor the non-asymptotic term in \eqref{eq:tensor_wick} into\small
\[
  \prod_{D\in \mcal{K}(\rho)} \bigg( \frac{1}{d!} \sum_{\pi \in \ubp([d-2]_D)}  \sum_{\phi: \pi \into [N]}  \prod_{B\in \pi} \left(\frac{\#(B)}{2}\right)!\prod_{v\in B \cap [d-2]_{D_L}} u_N^{(w_2^{v})}(\phi(B))\prod_{v' \in B \cap [d-2]_{D_R}} u_N^{(w_2^{v'})}(\phi(B)) \bigg).
\]\normalsize
It then suffices to show that\small
\begin{align*}
    &\frac{1}{d!} \sum_{\pi \in \ubp([d-2]_D)}  \sum_{\phi: \pi \into [N]}  \prod_{B\in \pi} \left(\frac{\#(B)}{2}\right)!\prod_{v\in B \cap [d-2]_{D_L}} u_N^{(w_2^{v})}(\phi(B))\prod_{v' \in B \cap [d-2]_{D_R}} u_N^{(w_2^{v'})}(\phi(B)) \\
    = &\frac{1}{d(d-1)} \left\langle u_N^{(i, 1)} \odot \cdots \odot u_N^{(i, d-2)}, u_N^{(i', 1)}\odot \cdots \odot u_N^{(i', d-2)} \right\rangle,
\end{align*}\normalsize
where the vectors $(u_N^{(i, j)})_{j \in [d-2]}$ (resp., $(u_N^{(i', j)})_{j \in [d-2]}$) correspond to the outer vertices $[d-2]_{D_L} = \{v_j\}_{j \in [d-2]}$ (resp., $[d-2]_{D_R} = \{v_j'\}_{j \in [d-2]}$).

To see this, we introduce the set $\ubp_2([d-2]_D)$ of uniform block permutations with blocks exclusively of size two, which can be thought of as the set of permutations $\ubp_2([d-2]_D) \cong \mfk{S}_{d-2}$. Note that for any $\pi \in \ubp([d-2]_D)$, 
\[
    \#\{\sigma \in \ubp_2([d-2]_D) : \sigma \leq \pi\} =  \prod_{B\in \pi} \left(\frac{\#(B)}{2}\right)!,
\]
where $\leq$ denotes the reversed refinement order. This allows us to reassemble the injective labels $\phi: \pi \into [N]$ for $\pi \in \ubp([d-2]_D)$ into general labels $\phi: \sigma \to [N]$ for $\sigma \in \ubp_2([d-2]_D)$. Indeed,\small
\begin{align*}
&\frac{1}{d!} \sum_{\pi \in \ubp([d-2]_D)} \sum_{\phi: \pi \into [N]} \prod_{B\in \pi} \left(\frac{\#(B)}{2}\right)!\prod_{v\in B \cap [d-2]_{D_L}} u_N^{(w_2^{v})}(\phi(B))\prod_{v' \in B \cap [d-2]_{D_R}} u_N^{(w_2^{v'})}(\phi(B)) \\
= &\frac{1}{d!} \sum_{\pi \in \ubp([d-2]_D)} \sum_{\substack{\sigma \in \ubp_2([d-2]_D) \\ \text{s.t. } \sigma \leq \pi}} \sum_{\phi: \pi \into [N]} \prod_{B\in \pi} \prod_{v\in B \cap [d-2]_{D_L}} u_N^{(w_2^{v})}(\phi(B))\prod_{v' \in B \cap [d-2]_{D_R}} u_N^{(w_2^{v'})}(\phi(B)) \\
= &\frac{1}{d!} \sum_{\sigma \in \ubp_2([d-2]_D)} \sum_{\phi: \sigma \to [N]} \prod_{B\in \sigma} \prod_{v\in B \cap [d-2]_{D_L}} u_N^{(w_2^{v})}(\phi(B))\prod_{v' \in B \cap [d-2]_{D_R}} u_N^{(w_2^{v'})}(\phi(B)) \\
= &\frac{1}{d!} \sum_{\sigma \in \mfk{S}_{d-2}} \prod_{j=1}^{d-2} \inn{u_N^{(i, j)}}{u_N^{(i', \sigma(j))}} \\
= &\frac{1}{d!}\per\left[\left(\inn{u_N^{(i, j)}}{u_N^{(i', k)}}\right)_{j, k= 1}^{d-2}\right] \\
= &\frac{1}{d(d-1)} \left\langle u_N^{(i, 1)} \odot \cdots \odot u_N^{(i, d-2)}, u_N^{(i', 1)}\odot \cdots \odot u_N^{(i', d-2)} \right\rangle,
\end{align*}\normalsize
where we have used the identification $\ubp_2([d-2]_D) \cong \mfk{S}_{d-2}$ and the permanental representation of the symmetrized inner product \cite[Theorem 2.2]{Min78}.
\end{proof}

\subsection{Concentration and almost sure convergence}\label{sec:almost_sure}

To prove concentration, we first need to define the random version of $\tau[G]$. For the graph $G$ in \eqref{eq:graph_of_tensors}, let
\[
  \trace(G) = \frac{1}{N^{\frac{m}{2}}} \sum_{\phi: V^\pi \to [N]} \prod_{w_1 \in W_1} \tdn{w_1}{\phi}
  \prod_{w_2\in W_2} \un{w_2}{\phi}
\]
so that $\tau[G] = \frac{1}{N} \E[\trace(G)]$. Proposition \ref{prop:convergenceinexpectation} already provides the asymptotics of $\frac{1}{N}\E[\trace(G)]$. We proceed to an analysis of the higher moments of the centered random variable $\trace(G)- \E[\trace(G)]$. We prove the following result in Section \ref{sec:concentration} using an extension of Lemma \ref{lem:supportontrees} to the disjoint union of multiple copies of $G$.
  
\begin{proposition}[Concentration]
\label{prop:concentration}
If $G$ is the graph in \eqref{eq:graph_of_tensors}, then
\begin{equation}\label{eq:concentration}
  \E\Bigg[\bigg(\frac{1}{N}\trace(G) - \E\Big[\frac{1}{N}\trace(G)\Big]\bigg)^{2M}\Bigg] = O_{m, d, M}(N^{-2M}).
\end{equation}
\end{proposition}

Propositions \ref{prop:convergenceinexpectation} and \ref{prop:concentration} allow us to prove our main results: Theorems \ref{thm:contracted_tensor_ensemble} and \ref{thm:concentrationofmixedmoments}.

\begin{proof}[Proof of Theorems \ref{thm:contracted_tensor_ensemble} and \ref{thm:concentrationofmixedmoments}]
Theorem \ref{thm:contracted_tensor_ensemble} follows from Theorem \ref{thm:concentrationofmixedmoments} via a standard application of the Borel-Cantelli lemma. To prove Theorem \ref{thm:concentrationofmixedmoments}, let $i_1, \dots, i_m\in I$ and recall that $\trace(G)= \trace(P_N(i_1, \dots, i_m))$ by construction. Proposition \ref{prop:concentration} implies
\[
 \varepsilon^{2M} \mathbb{P}\bigg[\bigg|\frac{1}{N}\trace(G) - \E\Big[\frac{1}{N}\trace(G)\Big]\bigg|>\varepsilon\bigg] = O_{m, d, M}(N^{-2M}).
\]
At the same time, we know that $\E\Big[\frac{1}{N}\trace(G)\Big] = \phi(s_{i_1}^{(N)}\cdots s_{i_m}^{(N)}) + O_{m, d}(N^{-1})$ from Proposition \ref{prop:convergenceinexpectation}. Using the triangle inequality, we conclude that
\[
 \varepsilon^{2M} \mathbb{P}\bigg[\bigg|\frac{1}{N}\trace(G) - \phi(s_{i_1}^{(N)}\cdots s_{i_m}^{(N)})\bigg|>\varepsilon'\bigg]=O_{m, d, M}(N^{-2M})
\]
for any $\varepsilon'>\varepsilon$, which proves \eqref{eq:concentration_of_mixed_moments}.
\end{proof}

\section{Proofs of the technical results}\label{sec:technical_results}

\subsection{Proof of Lemma \ref{lem:supportontrees}}\label{sec:tech_sec_one}

Recall that we must prove
\[
  \tau^0[G^\pi] = O_{m}(N^{-1})
\]
unless $H^\pi$ is a double tree and $\pim = \emptyset$. By our earlier discussion, we can restrict to partitions $\pi \in \mcal{P}(V)$ that satisfy \eqref{eq:centering_in_graph}.

Note that a diamond overlay $\mcal{B}_\pi(w_1) = \mcal{B}_\pi(w_1')$ does not imply that an inner vertex neighbor $v \in \mcal{B}(w_1) \cap V_{\text{in}}$ of $w_1$ gets identified with at least one other inner vertex neighbor $v' \in \mcal{B}(w_1') \cap V_{\op{in}}$ of $w_1'$; however, without such an identification, the equality of the neighborhoods can only hold if $v$ is identified with an outer vertex neighbor $v'' \in \mcal{B}(w_1') \cap V_{\text{out}}$. See Figure \ref{figure:overlay_square} and Figure \ref{figure:overlay_hexagon} for examples. This simple observation will be used frequently in the sequel, so we record it below.

\begin{lemma}\label{lem:mix_a_lot}
If $\mcal{B}_\pi(w_1) = \mcal{B}_\pi(w_1')$ and $v \in \mcal{B}(w_1) \cap V_{\op{in}}$ satisfies
\[
  [v]_\pi \neq [v']_\pi, \quad \forall v' \in \mcal{B}(w_1') \cap V_{\op{in}},
\]
then
\[
  \exists v'' \in \mcal{B}(w_1') \cap V_{\op{out}}: [v]_\pi = [v'']_{\pi}.
\]
In particular, $\pim \neq \emptyset$.
\end{lemma}

Condition \eqref{eq:centering_in_graph} forces every block $B = [v]_\pi \in \pi_{\op{out}}$ to have size $\#([v]_\pi) \geq 2$. The degrees of freedom in choosing $\phi|_{\pi_{\op{out}}}$ are then negated by the terms coming from the neighboring boxes. To see this, note that
\begin{align*}
  &\sum_{\phi: V^\pi \hookrightarrow [N]} \E\bigg[\prod_{w_1 \in W_1} \tdn{w_1}{\phi}\bigg] \prod_{w_2\in W_2} \un{w_2}{\phi}\\
  = &O_{m}\bigg(N^{\#(\pi_{\op{in}})}\sum_{\phi: \pi_{\op{mix}} \to [N]} \prod_{[v]_\pi \in \pi_{\op{mix}}} \prod_{\substack{w_2 \in W_2:\\ w_2 \sim [v]_\pi}} \Big|u_N^{(w_2)}(\phi([v]_\pi))\Big|\bigg) \\
  &\times O_{m}\bigg(\sum_{\phi': \pi_{\op{out}} \to [N]} \prod_{[v]_\pi \in \pi_{\op{out}}} \prod_{\substack{w_2 \in W_2:\\ w_2 \sim [v]_\pi}} \Big|u_N^{(w_2)}(\phi'([v]_\pi))\Big|\bigg),
\end{align*}
where we can control the last term\small
\[
  \sum_{\phi': \pi_{\op{out}} \to [N]} \prod_{[v]_\pi \in \pi_{\op{out}}} \prod_{\substack{w_2 \in W_2:\\ w_2 \sim [v]_\pi}} \Big|u_N^{(w_2)}(\phi'([v]_\pi))\Big| = \prod_{[v]_\pi \in \pi_{\op{out}}} \Big(\sum_{k \in [N]} \prod_{\substack{w_2 \in W_2:\\ w_2 \sim [v]_{\pi}}} |u_N^{(w_2)}(k)|\Big) \leq 1
\]\normalsize
by the Cauchy-Schwarz inequality and the fact that $|u_N^{(i, j)}(k)| \leq \norm{u_N^{(i, j)}}_2 = 1$. So, the $\pi_{\op{out}}$ vertices do not actually contribute to the exponent in the crude asymptotic $\eqref{eq:crude_asymptotic}$ and we have the refinement
\[
  \tau^0[G^\pi] =  O_m(N^{\#(\pi_{\op{in}}) + \#(\pi_{\op{mix}}) - \frac{m}{2} - 1}).
\]
In fact, we can apply the same argument to $\pi_{\op{mix}} = \pi_{\op{mix}}^{(1)} \sqcup \pi_{\op{mix}}^{(\geq 2)}$, where $\pi_{\op{mix}}^{(1)}$ is the set of mixed blocks such that each block has exactly one outer vertex and $\pi_{\op{mix}}^{(\geq 2)}$ is the set of mixed blocks such that each block has at least two outer vertices. This allows us to further reduce
\begin{align*}
  &O_m\bigg(N^{\#(\pi_{\op{in}})}\sum_{\phi: \pi_{\op{mix}} \to [N]} \prod_{[v]_\pi \in \pi_{\op{mix}}} \prod_{\substack{w_2 \in W_2:\\ w_2 \sim [v]_\pi}} \Big|u_N^{(w_2)}(\phi([v]_\pi))\Big|\bigg)\\
  = &O_m\bigg(N^{\#(\pi_{\op{in}})} \sum_{\phi: \pi_{\op{mix}}^{(1)} \to [N]} \prod_{[v]_\pi \in \pi_{\op{mix}}^{(1)}} \prod_{\substack{w_2 \in W_2:\\ w_2 \sim [v]_\pi}}
  \Big|u_N^{(w_2)}(\phi([v]_\pi))\Big|\bigg)\\
  &\times O_m\bigg(\sum_{\phi': \pi_{\op{mix}}^{(\geq 2)} \to [N]} \prod_{[v]_\pi \in \pi_{\op{mix}}^{(\geq 2)}} \prod_{\substack{w_2 \in W_2:\\ w_2 \sim [v]_\pi}} \Big|u_N^{(w_2)}(\phi'([v]_\pi))\Big| \bigg) \\
  = &O_m\bigg(N^{\#(\pi_{\op{in}})} \prod_{[v]_\pi \in \pi_{\op{mix}}^{(1)}} \prod_{\substack{w_2 \in W_2:\\ w_2 \sim [v]_{\pi}}} \inn{|u_N^{(w_2)}|}{\hat{1}_N}\bigg) = O_m(N^{\#(\pi_{\op{in}}) + \frac{\#(\pi_{\op{mix}}^{(1)})}{2}}),
\end{align*}
where $|u_N^{(w_2)}|$ is the entrywise absolute value of the contracting vector $u_N^{(w_2)}$ and $\hat{1}_N \in \R^N$ is the all-ones vector. The final form of our asymptotic is then
\[
  \tau^0[G^\pi] = O_{m}(N^{\#(\pi_{\op{in}}) + \frac{\#(\pi_{\op{mix}}^{(1)})}{2} - \frac{m}{2} - 1}),
\]
which is controlled by
\begin{proposition}\label{prop:double_tree}
For any partition $\pi$ satisfying \eqref{eq:centering_in_graph},
\begin{equation}\label{eq:main_inequality}
  \#(\pi_{\op{in}}) + \frac{\#(\pi_{\op{mix}}^{(1)})}{2} \leq \frac{m}{2} + 1
\end{equation}
with equality iff $m$ is even, $\pi_{\op{mix}} = \emptyset$, and $H^\pi$ is a double tree. In all other cases,
\[
  \#(\pi_{\op{in}}) + \frac{\#(\pi_{\op{mix}}^{(1)})}{2} \leq \frac{m}{2}.
\]
\end{proposition}

\begin{proof}
First, we introduce some auxiliary sets:
\begin{align*}
    \pi_{\op{mix}}^{(1, 1)} &= \{B \in \pi : \#(B \cap \vout) = 1, \#(B \cap \vin) = 1\}; \\
    \pi_{\op{mix}}^{(1, \geq 2)} &= \{B \in \pi : \#(B \cap \vout) = 1, \#(B \cap \vin) \geq 2\}; \\
    \pi_{\op{in}}^{(1)} &= \{B \in \pi : \#(B \cap \vout) = 0, \#(B \cap \vin) = 1\}; \\
    \pi_{\op{in}}^{(\geq 2)} &= \{B \in \pi : \#(B \cap \vout) = 0, \#(B \cap \vin) \geq 2\};
\end{align*}
and
\begin{align*}
V_\pi^{(1, 1)} &= \cup_{B \in \pi_{\op{mix}}^{(1, 1)}} B \cap \vin; \\
V_\pi^{(1, \geq 2)} &= \cup_{B \in \pi_{\op{mix}}^{(1, \geq 2)}} B \cap \vin; \\
V_\pi^{(0, 1)} &= \cup_{B \in \pi_{\op{in}}^{(1)}} B \cap \vin; \\
V_\pi^{(0, \geq 2)} &= \cup_{B \in \pi_{\op{in}}^{(\geq 2)}} B \cap \vin.
\end{align*}
For a diamond $w_1 \in W_1$, we use the notation $\mcal{O}(w_1) = \{v \in \vout: v \sim w_1\}$ for its set of outer vertex neighbors in $G$.

Note that we only need to prove the only if portion of the iff statement, the converse following from the classical $d = 2$ case. We prove Proposition \ref{prop:double_tree} by induction on $m$ with the base cases $m=1, 2$. If $m = 1$, then \eqref{eq:centering_in_graph} implies that $\#(\pi) = 1$.  In particular, $\#(\pi_{\op{in}}) = 0$ and $\#(\pi_{\op{mix}}^{(1)}) \leq 1$, whence $\#(\pi_{\op{in}}) + \frac{\#(\pi_{\op{mix}}^{(1)})}{2} \leq \frac{1}{2}$.

If $m = 2$, then $\#(\pi_{\op{mix}}^{(1)}) = 0$. Indeed, an internal vertex $v \in B \in \pi_{\op{mix}}$ is necessarily adjacent to the two diamonds $w_1$ and $w_2$. Condition \eqref{eq:centering_in_graph} imposes the equality
\[
    \#(\{e: [v]_\pi \stackrel{e}{\sim} w_1\}) = \#(\{e: [v]_\pi \stackrel{e}{\sim} w_2\}),
\]
which implies that $\#(B \cap \vout)$ is even and not equal to 1. Now, if $\#(\pi_{\op{in}}) \leq 1$, then we are done. Otherwise, $\#(\pi_{\op{in}}) = 2$, in which case $\#(\pi_{\op{in}}) + \frac{\#(\pi_{\op{mix}}^{(1)})}{2} = 2$. Since $m = \#(\pi_{\op{in}}) = 2$, we have no inner vertices left to mix $\pi_{\op{mix}} = \emptyset$ and $H^\pi$ is the double tree consisting of two vertices with two edges between them.

Now suppose that $m \geq 3$. Note that $\pi_{\op{mix}}^{(\geq 2)}$ and $\pi_{\op{out}}$ do not play a role in the statement, so we can assume that 
\begin{equation}\label{eq:merged_blocks}           \#(\pi_{\op{mix}}^{(\geq 2)} \cup \pi_{\op{out}}) \leq 1
\end{equation}
by merging every such block. The resulting partition clearly still satisfies \eqref{eq:centering_in_graph}. If $\#(\pi_{\op{in}}^{(1)}) \leq \#(\pi_{\op{mix}}^{(1, \geq 2)})$, then we are done. To see this, note that every block $B \in \pi_{\op{mix}}^{(1, \geq 2)} \cup \pi_{\op{in}}^{(\geq 2)}$ satisfies $\#(B) \geq 2$, whence 
\begin{equation}\label{eq:singleton_reduction}
\begin{aligned}
    \#(\pi_{\op{in}}) + \frac{\#(\pi_{\op{mix}}^{(1)})}{2} &= \#(\pi_{\op{in}}^{(1)}) + \frac{\#(\pi_{\op{mix}}^{(1, \geq 2)})}{2} + \#(\pi_{\op{in}}^{(\geq 2)}) + \frac{\#(\pi_{\op{mix}}^{(1, 1)})}{2} \\
    &\leq \frac{\#(\pi_{\op{in}}^{(1)})}{2} + \#(\pi_{\op{mix}}^{(1, \geq 2)}) + \#(\pi_{\op{in}}^{(\geq 2)}) + \frac{\#(\pi_{\op{mix}}^{(1, 1)})}{2} \\
    &\leq \frac{\#(V_\pi^{(0, 1)})}{2} + \frac{\#(V_\pi^{(1, \geq 2)})}{2} + \frac{\#(V_\pi^{(0, \geq 2)})}{2} + \frac{\#(V_\pi^{(1, 1)})}{2} \\
    &\leq \frac{\#(\vin)}{2} = \frac{m}{2}.
\end{aligned}
\end{equation}

Henceforth, we assume that
\begin{equation}\label{eq:singleton_inequality}
    \#(\pi_{\op{in}}^{(1)}) > \#(\pi_{\op{mix}}^{(1, \geq 2)}).
\end{equation}
We will need the following result.

\begin{lemma}\label{lem:simple_edge}
For any block $B \in \pi_{\op{mix}}^{(1, \geq 2)}$, there is at most one block $B' \in \pi_{\op{in}}^{(1)}$ connected to $B$ in $H^\pi$ with a simple edge.
\end{lemma}

\begin{proof}[Proof of Lemma \ref{lem:simple_edge}]
Let $B \sim_e B'$ be as above. Since $B' \in \pi_{\op{in}}^{(1)}$, there is exactly one other edge $e' \neq e$ incident to $B'$ in $H^\pi$. By assumption, $e'$ is not incident to $B$. Condition \eqref{eq:centering_in_graph} then forces $\{e, e'\} \in \sigma_\pi$, which can only be achieved if $B \cap \mcal{O}(e') \neq \emptyset$. We conclude that $e'$ is responsible for the lone outer vertex in $B \in \pi_{\op{mix}}^{(1, \geq 2)}$.

The requirement $e \sim_{\sigma_\pi} e'$ precludes the possibility of $e'$ being incident to another block $\tilde{B}' \neq B'$ in $\pi_{\op{in}}^{(1)}$. So, if $\tilde{B}' \in \pi_{\op{in}}^{(1)}$ is also connected to $B$ in $H^\pi$ with a simple edge $\tilde{e}$ and $\tilde{e}'$ is the only other edge incident to $\tilde{B}'$ in $H^\pi$, it must be that $\tilde{e}' \neq e'$. We can then repeat the argument above and arrive at the contradiction that $\tilde{e}'$ is responsible for another outer vertex $B \cap \mcal{O}(\tilde{e}') \neq \emptyset$ in $B \in \pi_{\op{mix}}^{(1, \geq 2)}$.
\end{proof}

Assumption \eqref{eq:singleton_inequality} and Lemma \ref{lem:simple_edge} guarantee the existence of a block $B_0 \in \pi_{\op{in}}^{(1)}$ that is not connected to any element of $\pi_{\op{mix}}^{(1, \geq 2)}$ with a simple edge. Let $e_1, e_2$ be the two edges incident to $B_0$. We write $v_1$ and $v_2$ for the vertices incident to $e_1$ and $e_2$ in $H$ respectively with $B_1 = [v_1]_\pi$ and $B_2 = [v_2]_\pi$ the corresponding blocks. Since $B_0 \in \pi_{\op{in}}^{(1)}$, there are no other edges incident to $B_0$ and \eqref{eq:centering_in_graph} imposes the constraint $\{e_1, e_2\} \in \sigma_\pi$.

We claim that $B_1 \not\in \pi_{\op{mix}}^{(1, 1)}$. Otherwise, $B_1 \in \pi_{\op{mix}}^{(1, 1)}$ and $B_1 \neq B_2$ since $v_1 \neq v_2$ are distinct inner vertices (recall that $m \geq 3$). The constraint $e_1 \sim_{\sigma_\pi} e_2$ can then only be fulfilled if $\#(B_1 \cap \mcal{O}(e_2)) = 1$. In that case, the other edge $e_3 \neq e_2$ incident to $v_1$ in $H$ ($e_3 \neq e_2$ since $v_1 \neq v_2$) cannot satisfy \eqref{eq:centering_in_graph}: $\#(\mcal{B}_\pi(e_3)) \neq 1$ since $B_1 \in \pi_{\op{mix}}^{(1, 1)}$, and $\mcal{B}_\pi(e_3) \neq \mcal{B}_\pi(e_1), \mcal{B}_\pi(e_2)$ since $B_0 \in \pi_{\op{in}}^{(1)}$.

We then have the following trichotomy:
\begin{enumerate}
\item $B_1 \in \pi_{\op{in}}$, in which case $B_1 = B_2$ since $e_1 \sim_{\sigma_\pi} e_2$;
\item $B_1 \in \pi_{\op{mix}}^{(1, \geq 2)}$, in which case $B_1 = B_2$ because $B_0$ is not connected to any element of $\pi_{\op{mix}}^{(1, \geq 2)}$ by a simple edge by assumption (furthermore, the same parity argument as before shows that $B_1 \cap \mcal{O}(e_1) = B_1 \cap \mcal{O}(e_2) = \emptyset$); 
\item $B_1 \in \pi_{\op{mix}}^{(\geq 2)}$, in which case $B_2 \not\in \pi_{\op{in}} \cup \pi_{\op{mix}}^{(1, \geq 2)}$: otherwise, interchanging $B_1$ and $B_2$, the analysis above would imply that $B_1 = B_2 \in \pi_{\op{in}} \cup \pi_{\op{mix}}^{(1, \geq 2)}$. Thus, it must be that $B_2 \in \pi_{\op{mix}}^{(\geq 2)}$ as well, whence $B_1 = B_2$ by our earlier reduction \eqref{eq:merged_blocks}.
\end{enumerate}
In any case, $B_1 = B_2$. We construct a smaller graph $\tilde{G}$ from $G$ by removing: the lone vertex in $B_0$, the adjacent diamonds $e_1, e_2$, the outer vertices $\mcal{O}(e_1) \cup \mcal{O}(e_2)$, and every incident edge. We also merge $v_1$ and $v_2$. In other words, we have simply pinched off a small portion of the decorated cycle \eqref{eq:graph_of_tensors}, bringing us to the case of $m-2 \geq 1$.

By a slight abuse of notation, we write $\tilde{\pi} = \pi|_{\tilde{G}}$ for the natural partition of $\tilde{G}$ induced by $\pi$ (we merged $v_1$ and $v_2$ in constructing $\tilde{G}$, but $[v_1]_\pi = B_1 = B_2 = [v_2]_\pi$ anyway). Note that the partition $\tilde{\pi}$ still satisfies \eqref{eq:centering_in_graph} since the other diamonds are unaffected (recall that $\{e_1, e_2\} \in \sigma_\pi$). We can then apply the induction hypothesis, which tells us that
\begin{equation}\label{eq:induction_hypothesis}
    \#(\tilde{\pi}_{\op{in}}) + \frac{\#(\tilde{\pi}_{\op{mix}}^{(1)})}{2} \leq \frac{m}{2}
\end{equation}
with equality iff $m-2$ is even, $\tilde{\pi}_{\op{mix}} = \emptyset$, and $\tilde{H}^{\tilde{\pi}}$ is a double tree. 

To see what this says about our original partition $\pi$, consider a block $B \in \pi_{\op{mix}}^{(1)}$. Since $e_1 \sim_{\sigma_\pi} e_2$, we know that $B \cap \mcal{O}(e_1) = B \cap \mcal{O}(e_2) = \emptyset$ by the now familiar parity argument. So, after removing $\mcal{O}(e_1) \cup \mcal{O}(e_2)$, the block $B$ becomes a block $\tilde{B} \in \tilde{\pi}_{\op{mix}}^{(1)}$, leading to the inequality
\begin{equation}\label{eq:mixed_tilde}
    \#(\pi_{\op{mix}}^{(1)}) \leq \#(\tilde{\pi}_{\op{mix}}^{(1)}).
\end{equation}
At the same time, removing $\mcal{O}(e_1) \cup \mcal{O}(e_2)$ creates a block in $\tilde{\pi}_{\op{in}}$ for each block $B \in \pi_{\op{mix}}$ such that $B \cap \vout \subset \mcal{O}(e_1) \cup \mcal{O}(e_2)$. We also lose a block in $\pi_{\op{in}}$ from the removal of $B_0 \in \pi_{\op{in}}^{(1)}$, and so
\[
    \#(\tilde{\pi}_{\op{in}}) = \#(\pi_{\op{in}}) - 1 + \#(\{B \in \pi_{\op{mix}} : B \cap \vout \subset \mcal{O}(e_1) \cup \mcal{O}(e_2)\}).
\]
In any case, $\#(\pi_{\op{in}}) \leq \#(\tilde{\pi}_{\op{in}}) + 1$. Combining this with \eqref{eq:induction_hypothesis} and \eqref{eq:mixed_tilde}, we obtain
\[ 
    \#(\pi_{\op{in}}) + \frac{\#(\pi_{\op{mix}}^{(1)})}{2} \leq \frac{m}{2} + 1
\]
with $\#(\pi_{\op{in}}) + \frac{\#(\pi_{\op{mix}}^{(1)})}{2} \leq \frac{m}{2}$ unless $\tilde{\pi}_{\op{mix}} = \emptyset$, $\tilde{H}^{\tilde{\pi}}$ is a double tree, and $\#(\{B \in \pi_{\op{mix}} : B \cap \vout \subset \mcal{O}(e_1) \cup \mcal{O}(e_2)\}) = 0$. 

Now assume that the latter conditions hold. Since $\tilde{\pi}_{\op{mix}} = \emptyset$, it follows from \eqref{eq:mixed_tilde} that $\pi_{\op{mix}}^{(1)} = \emptyset$. Similarly, $\tilde{\pi}_{\op{mix}} = \emptyset$ and $\#(\{B \in \pi_{\op{mix}} : B \cap \vout \subset \mcal{O}(e_1) \cup \mcal{O}(e_2)\}) = 0$ imply that $\pi_{\op{mix}}^{(\geq 2)} = \emptyset$, and so $\pi_{\op{mix}} = \emptyset$. Finally, $H^\pi$ is obtained from $\tilde{H}^{\tilde{\pi}}$ by adding the vertex $B_0$ and the double edge $\{e_1, e_2\}$ between $B_0$ and $B_1 = B_2$: if $\tilde{H}^{\tilde{\pi}}$ is a double tree, then so too is $H^\pi$.  
\end{proof}

\subsection{Proof of Proposition \ref{prop:concentration}}\label{sec:concentration}

The idea of the proof is to realize \eqref{eq:concentration} as a sum over vertex labelings of the disjoint union of $2M$ copies of the graph of tensors $G$ in \eqref{eq:graph_of_tensors} satisfying certain properties. We will need to adapt the ideas in the previous section in conjunction with those in \cite[Section 3.2]{Au18} to bound the number of admissible vertex labelings in this extended picture. The gain in the exponent compared to the single graph case of Section \ref{sec:tech_sec_one} comes from the fact that one now needs an overlay of diamonds across different copies of $G$ to prevent the vanishing of the expectation due to the centering: as we will see, such identifications are always suboptimal.

We start with some notation. Let $G_k = (V_k, W_1^{(k)} \sqcup W_2^{(k)}, E_k)$ be a graph of tensors \eqref{eq:graph_of_tensors} of length $m_k$. The statement of Proposition \ref{prop:concentration} assumes that $m_k \equiv m$, but we will prove a more general result that is amenable to induction. For a vertex labeling $\phi_k: V_k \to [N]$, we define the centered random variable\small
\begin{equation}\label{eq:phi_product}
  Q_{\phi_k} = \bigg(\prod_{w_1 \in W_1^{(k)}} \mbf{T}_{d, N}(w_1 | \phi_k) - \E\Big[\prod_{w_1 \in W_1^{(k)}} \mbf{T}_{d, N}(w_1 | \phi_k)\Big]\bigg)\prod_{w_2 \in W_2^{(k)}} u_N^{(w_2)}(w_2 | \phi_k).
\end{equation}\normalsize
This allows us to write the expectation of a product of unnormalized traces as
\begin{equation}\label{eq:K_moment}
  \E\Big[\prod_{k = 1}^K \big(\trace(G_k) - \E[\trace(G_k)]\big)\Big] 
  = N^{-\frac{\sum_{k=1}^K m_k}{2}} \sum_{\substack{(\phi_1, \ldots, \phi_K) \\ \text{s.t. } \phi_k: V_k \to [N]}} \E\bigg[\prod_{k \in [K]} Q_{\phi_k}\bigg].
\end{equation}
We say that vertex labels $\phi_k, \phi_{k'}$ are \emph{matched} if there exist vertices $(w_1, w_1') \in W_1^{(k)} \times W_1^{(k')}$ such that
\[
  \phi_k(\mcal{B}(w_1)) = \phi_{k'}(\mcal{B}(w_1')), 
\]
where the equality is in the sense of multisets. In words, this says that there must be two diamonds such that the $\phi_k$-labeled neighborhood of the first matches the $\phi_{k'}$-labeled neighborhood of the second. We say that a coordinate $\phi_{k'}$ in a $K$-tuple $(\phi_1, \ldots, \phi_K)$ is \emph{unmatched} if it is not matched to any other coordinate $\phi_{k}$ with $k \neq k'$, in which case
\[
  \E\bigg[\prod_{k \in [K]} Q_{\phi_k}\bigg] = \E\bigg[\prod_{\substack{k \in [K]: \\ k \neq k'}} Q_{\phi_k}\bigg]\E[Q_{\phi_{k'}}] = 0
\]
by the independence of our tensor entries and the centeredness of $Q_{\phi_{k'}}$. We may then restrict the sum in \eqref{eq:K_moment} to $K$-tuples with no unmatched coordinates.

To translate this condition to the setting of partitions and injective labelings, we will need some additional notation. Let $\mcal{G} = (\mcal{V}, \mcal{W}, \mcal{E})$ be the disjoint union $\sqcup_{k = 1}^K G_k$ of the graphs $G_k$. In particular,
\begin{align}
  \mcal{V} = \sqcup_{k = 1}^{K} V_k; \label{eq:disjoint_union_vertices}\\
  \mcal{W}_1 = \sqcup_{k = 1}^{K} W_1^{(k)};\\
  \mcal{W}_2 = \sqcup_{k = 1}^{K} W_2^{(k)};\\
  \mcal{W} = \mcal{W}_1 \sqcup \mcal{W}_2;\\
  \mcal{E} = \sqcup_{k = 1}^{K} E_k \label{eq:disjoint_union_edges}.
\end{align}
This allows us to rewrite \eqref{eq:K_moment} as
\[
  \E\Big[\prod_{k = 1}^K \big(\trace(G_k) - \E[\trace(G_k)]\big)\Big] 
  = N^{-\frac{\sum_{k=1}^K m_k}{2}} \sum_{\pi \in \mcal{P}(\mcal{V})} \sum_{\Phi: \mcal{V}^\pi \hookrightarrow [N]} \E\bigg[\prod_{k \in [2M]} Q_{\Phi|_{V_k}}\bigg],
\]
where, by a slight abuse of notation, we have used the fact that a map $\Phi: \mcal{V}^\pi \hookrightarrow [N]$ defines a vertex labeling $\Phi|_{V_k}: V_k \to [N]$ for every $k \in [K]$. The condition about matchings from the previous paragraph then allows us to restrict to partitions $\pi$ such that every copy $G_k$ has a diamond overlay with at least one other copy $G_{k'}$ with $k' \neq k$.

The diamond overlays between the $(G_k)_{k=1}^K$ define an equivalence relation on $[K]$ with $r_\pi \leq \lfloor\frac{K}{2}\rfloor$ equivalence classes $A_1, \ldots, A_{r_\pi}$. In particular, elements $k, k' \in [K]$ belong to the same equivalence class $A_q$ iff there are subgraphs $G_{k_0}, \ldots, G_{k_t}$ in $\mcal{G}$ with $k_0 = k$ and $k_t = k'$ such that $G_{k_s}$ has a diamond overlay with $G_{k_{s+1}}$ for $s = 0, \ldots, t-1$. This allows us to factor
\[
   \E\bigg[\prod_{k \in [K]} Q_{\Phi|_{V_k}}\bigg] = \prod_{q \in [r_\pi]} \E\bigg[\prod_{k \in A_q} Q_{\Phi|_{V_k}}\bigg].
\]

 We need the obvious analogue of Definition \ref{defn:H_graph} for $\mcal{G}$, which we denote by $\mcal{H} = \sqcup_{k = 1}^K H_k$ and $\mcal{H}^\pi = \mcal{H}^{\pi\restriction_{\msr{V}(\mcal{H})}}$. We write $\mcal{H}_q^\pi$ for the subgraph spanned by the edges of $(H_k)_{k \in A_q}$ in $\mcal{H}^\pi$. We decompose the vertex set of $\mcal{H}_q^\pi$ as before:
\[
  \msr{V}(\mcal{H}_q^\pi) = \pi_{\op{in}, q} \sqcup \pi_{\op{mix}, q}^{(1)} \sqcup \pi_{\op{mix}, q}^{(\geq 2)}.
\]
Recall that the contracting vectors are deterministic: this allowed us to factor out their contribution in the formulation \eqref{eq:phi_product} of $Q_\phi$. So, we can repeat the analysis from Section \ref{sec:tech_sec_one} to conclude that\small
\[
 \sum_{\Phi: \mcal{V}^\pi \hookrightarrow [N]} \prod_{q \in [r_\pi]} \E\bigg[\prod_{k \in A_q} Q_{\Phi|_{V_k}}\bigg] = O_{m, M}\Big(N^{\sum_{q \in [r_\pi]} \#(\pi_{\op{in}, q}) + \frac{\#(\pi_{\op{mix}, q}^{(1)})}{2}}\Big).
\]\normalsize
So, we will be done if we can prove that
\begin{equation}\label{eq:equivalence_class_count}
  \#(\pi_{\op{in}, q}) + \frac{\#(\pi_{\op{mix}, q}^{(1)})}{2} \leq \frac{\sum_{k \in A_q} m_k}{2}.
\end{equation}
Note that \eqref{eq:equivalence_class_count} is localized to an equivalence class $A_q$. So, without loss of generality, we may assume that there is only one equivalence class $A_q = [K]$. Thus, we need to prove

\begin{proposition}\label{prop:butterfly_obstruction}
Let $\mcal{G} = (\mcal{V}, \mcal{W}, \mcal{E})$ be the disjoint union as in \eqref{eq:disjoint_union_vertices}-\eqref{eq:disjoint_union_edges} with $K \geq 2$. If $\pi \in \mcal{P}(\mcal{V})$ satisfies
\begin{align}
    \#(\mcal{B}_\pi(w_1)) = 1 \text{ \emph{or} } \exists w_1' \in \mcal{W}_1\setminus\{w_1\}: \mcal{B}_\pi(w_1) &= \mcal{B}_\pi(w_1'), \qquad \forall w_1 \in \mcal{W}_1; \label{eq:centering_in_graph_redux}\\
    \exists k' \neq k, (w_1, w_1') \in W_1^{(k)} \times W_1^{(k')}: \mcal{B}_\pi(w_1) &= \mcal{B}_\pi(w_1') \label{eq:cross_overlay}, \qquad \forall k \in [K],
\end{align}
then
\[
    \#(\pi_{\op{in}}) + \frac{\#(\pi_{\op{mix}}^{(1)})}{2} \leq \frac{\sum_{k=1}^K m_k}{2}.
\]
\end{proposition}

\begin{proof}
Before getting started, we note that \eqref{eq:centering_in_graph_redux} is the analogue of \eqref{eq:centering_in_graph} in this context: here, the overlaying diamond $w_1'$ might come from a different graph of tensors $G_{k} \neq G_{k'}$. The additional constraint on the partition \eqref{eq:cross_overlay} is the mandatory crossed overlay necessary to survive the centeredness of the $Q_{\Phi|_{V_k}}$ as discussed earlier.

We adapt the proof of Proposition \ref{prop:double_tree} and proceed by induction on $m = \sum_{k = 1}^K m_k$ with the base cases $m = 2, 3$. The case of $m = 2$ corresponds to $K = 2$ and $m_1 = m_2 = 1$, meaning $G_1$ and $G_2$ each consist of a single diamond with a double edge to a single inner vertex and simple edges to $d-2$ outer vertices. The crossed overlay \eqref{eq:cross_overlay} bounds $\#(\pi_{\op{in}}) \leq 1$ with $\#(\pi_{\op{in}}) = 1$ only if $\#(\pi_{\op{mix}}^{(1)}) = 0$. If $\#(\pi_{\op{in}}) = 0$, then the fact that each $G_i$ has a double edge to its lone inner vertex means that each mixed block must have at least two outer vertices, and so $\#(\pi_{\op{mix}}^{(1)}) = 0$.

There are two possibilities for the case of $m = 3$. First, it could be that $K = 3$ and $m_1 = m_2 = m_3 = 1$. This is treated almost identically to the $m = 2$ case. Second, it could be that $K = 2$, $m_1 = 1$, and $m_2 =2$, meaning $G_1$ is as before and $G_2$ is as in Figure \ref{figure:covariance}. Let $B \in \pi$ be the block containing the inner vertex in $G_1$. If $\#(B \cap \mcal{V}_{\op{in}}) = 3$, then $\#(\pi_{\op{in}}) + \#(\pi_{\op{mix}}^{(1)}) \leq 1$ and we are done. If $\#(B \cap \mcal{V}_{\op{in}}) = 2$, then \eqref{eq:cross_overlay} imposes the constraint $B \in \pi_{\op{mix}}$ because of the double edge in $G_1$. But then $\#(\pi_{\op{in}}) \leq 1$ since there is only one unaccounted inner vertex left in $\mcal{G}$ and again we are done. Finally, if $\#(B \cap \mcal{V}_{\op{in}}) = 1$, then \eqref{eq:cross_overlay} imposes the constraint $B \in \pi_{\op{mix}}^{(\geq 2)}$ (again, because of the double edge in $G_1$). In that case, $\#(\pi_{\op{in}}) \leq 2$ and we will be done if we can rule out equality. Since $\#(B \cap \mcal{V}_{\op{in}}) = 1$, the crossed overlay in \eqref{eq:cross_overlay} keeps the inner vertices of the two overlaid diamonds in separate blocks. This forces the inner vertices of each diamond to be merged with outer vertices, and so $\#(\pi_{\op{in}}) = 0$.
  
Now suppose that $m \geq 4$. From here, the proof largely resembles Proposition \ref{prop:double_tree} and we simply outline the argument. As before, we can assume \eqref{eq:merged_blocks} by merging every such block and noting that the resulting partition still satisfies \eqref{eq:centering_in_graph_redux}-\eqref{eq:cross_overlay}. Similarly, we can assume \eqref{eq:singleton_inequality} since otherwise we are done using the same line of reasoning as in \eqref{eq:singleton_reduction}. Lemma \ref{lem:simple_edge} still holds for $\mcal{H}^\pi$ with the same proof: together with \eqref{eq:singleton_inequality}, this guarantees the existence of a block $B_0 \in \pi_{\op{in}}^{(1)}$ that is not connected to any element of $\pi_{\op{mix}}^{(1, \geq 2)}$ with a simple edge as before.

Assume that the lone vertex $v_0 \in B_0 \in \pi_{\op{in}}^{(1)}$ comes from $H_{k_0}$. Here, we need to be careful: in the proof of Proposition \ref{prop:double_tree}, we knew that $m_{k_0} \geq 3$ because of the induction on $m_{k_0}$; now, we are inducting on $m = \sum_{k = 1}^K m_k$, and so we cannot immediately assume $m_{k_0} \geq 3$. We can rule out $m_{k_0} = 1, 2$ as follows. If $m_{k_0} = 1$, then there is a unique inner vertex in $G_{k_0}$, necessarily $v_0$, and the crossed overlay in \eqref{eq:cross_overlay} forces $B_0 \in \pi_{\op{in}}^{(\geq 2)} \cup \pi_{\op{mix}}$, a contradiction. Similarly, if $m_{k_0} = 2$, then $G_{k_0}$ is as in Figure \ref{figure:covariance} and the fact that $B_0 \in \pi_{\op{in}}^{(1)}$ precludes a crossed overlay \eqref{eq:cross_overlay} with any other $G_k$, again a contradiction.

We can now repeat the argument leading to the trichotomy in Proposition \ref{prop:double_tree}, pinch off the edges incident to $v_0$ in $H_{k_0}$, and reduce to the case of $\tilde{\pi}$ and $m - 2 \geq 2$, where the induction hypothesis tells us that
\[
    \#(\tilde{\pi}_{\op{in}}) + \frac{\#(\tilde{\pi}_{\op{mix}}^{(1)})}{2} \leq \frac{m-2}{2}.
\]
As before, we have the inequalities $\#(\pi_{\op{mix}}^{(1)}) \leq \#(\tilde{\pi}_{\op{mix}}^{(1)})$ and  $\#(\pi_{\op{in}}) \leq \#(\tilde{\pi}_{\op{in}}) + 1$, which allow us to conclude.
\end{proof}

\begin{figure}
\begingroup%
  \makeatletter%
  \providecommand\color[2][]{%
    \errmessage{(Inkscape) Color is used for the text in Inkscape, but the package 'color.sty' is not loaded}%
    \renewcommand\color[2][]{}%
  }%
  \providecommand\transparent[1]{%
    \errmessage{(Inkscape) Transparency is used (non-zero) for the text in Inkscape, but the package 'transparent.sty' is not loaded}%
    \renewcommand\transparent[1]{}%
  }%
  \providecommand\rotatebox[2]{#2}%
  \newcommand*\fsize{\dimexpr\f@size pt\relax}%
  \newcommand*\lineheight[1]{\fontsize{\fsize}{#1\fsize}\selectfont}%
  \ifx\svgwidth\undefined%
    \setlength{\unitlength}{360bp}%
    \ifx\svgscale\undefined%
      \relax%
    \else%
      \setlength{\unitlength}{\unitlength * \real{\svgscale}}%
    \fi%
  \else%
    \setlength{\unitlength}{\svgwidth}%
  \fi%
  \global\let\svgwidth\undefined%
  \global\let\svgscale\undefined%
  \makeatother%
  \begin{picture}(1,0.375)%
    \lineheight{1}%
    \setlength\tabcolsep{0pt}%
    \put(0,0){\includegraphics[width=\unitlength,page=1]{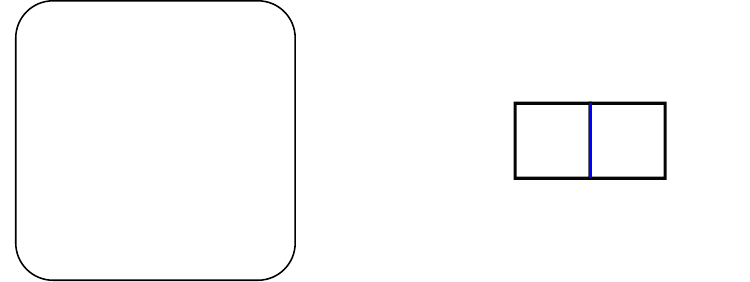}}%
    \put(0.19862893,1.96292302){\color[rgb]{0,0,0}\makebox(0,0)[lt]{\begin{minipage}{0.14438896\unitlength}\raggedright \end{minipage}}}%
    \put(0,0){\includegraphics[width=\unitlength,page=2]{fig_copies.pdf}}%
    \put(0.48613872,0.17887883){\color[rgb]{0,0,0}\makebox(0,0)[lt]{\lineheight{1.66666675}\smash{\begin{tabular}[t]{l}$\mapsto$\end{tabular}}}}%
  \end{picture}%
\endgroup%

\caption{An example of the double tree obstruction for $\mcal{H}^\pi$. For convenience, we have omitted the outer vertices. Even if one assumes $\pi_{\op{mix}} = \emptyset$, one cannot identify the remaining vertices in such a way that will produce a double tree for $\mcal{H}^\pi$.
}\label{figure:copies}
\end{figure}

\begin{remark}\label{rem:double_tree}
The suboptimality of the crossed overlay \eqref{eq:cross_overlay} and the subsequent gain in \eqref{eq:equivalence_class_count} versus \eqref{eq:main_inequality} comes from the fact that the condition for equality in Proposition \ref{prop:double_tree} cannot be achieved in the setting of Proposition \ref{prop:butterfly_obstruction} for $\mcal{H}^\pi$. See Figure \ref{figure:copies} for an illustration.
\end{remark}

Setting $K = 2M$, $m_k \equiv m$, and $G_k \equiv G$ in Proposition \ref{prop:butterfly_obstruction}, it follows that
\[
  \E\Big[\big(\trace(G) - \E[\trace(G)]\big)^{2M}\Big] = O_{m, d, M}(1),
\]
as was to be shown.

\subsubsection*{Acknowledgments}
The authors would like to thank the anonymous referees for their detailed feedback. In particular, the authors are grateful to the anonymous referee who suggested a shorter proof of Proposition \ref{prop:double_tree} but insisted on remaining anonymous.

\bibliographystyle{amsalpha}
\bibliography{master_bib}

\end{document}